\documentclass[12pt,reqno,a4paper]{memoir} 

\usepackage{thesis}

\begin{document}

\frontmatter
\titleen
\subtitleen
\titlelt
\subtitlelt
\tableofcontents*
\chapter*{Notation}

\noindent
\begin{tabular}[h]{c l}
$\N$ & the set of positive integers\\
$\N_0$ & the set of nonnegative integers\\
$\Z$ & the set of integers\\
$\Q$ & the set of rational numbers\\
$\R$ & the set of real numbers\\
$\C$ & the set of complex numbers\\
$\overline{l,m}$ & the set $\{l,l+1,\ldots,m\}$\\
$\S_n$ & the symmetric group on a finite set of $n$ symbols\\
$\sigma$ & generic element of the $\S_n$\\
$\kk_j(\sigma )$ & the number of cycles of length $j$ in a permutation $\sigma$\\
$\bar{s}$ & generic vector $(s_1,\ldots,s_n)\in\N_0^n$\\
$\ell(\bar{s})$ & the sum $1s_1+2s_2+\ldots+ns_n$\\
$\ell_r(\bar{s})$ & the sum $1s_1+2s_2+\ldots+rs_r$\\
$\ell_{rn}(\bar{s})$ & the sum $(r+1)s_{r+1}+(r+2)s_{r+2}+\ldots+ns_n$\\
$\log$ & the natural logarithm\\
$\Gamma$ & the Gamma function\\
$\Pr$ & the probability\\
$\E X$ & the expected value of a random variable $X$ \\
$\re$ & the Euler number\\
$\gamma$ & the Euler–Mascheroni constant\\
$H_r$ & the $r$-th harmonic number\\
$x$ & the positive solution to the equation $\sum_{j=1}^rx^j=n$\\
$\xi(v)$ & the positive solution to the equation $\re^{\xi(v)}=1+v\xi(v)$\\
$I(v)$ & the function $\int_0^v(\re^t-1)t^{-1}\, dt$\\
$\varrho(v)$ & the Dickman function;\\ & $\varrho(v)=1$ if $v\in[0,1]$, and $v\varrho'(v)+\varrho(v-1)=0$ if $v>1$\\
$\omega(v)$ & the Buchstab function;\\ & $v\omega(v)=1$ if $v\in[1,2]$, and $(v\omega(v))'=\omega(v-1)$ if $v>2$\\
\end{tabular}

\mainmatter
\chapter{Introduction}\label{Chapter1}

\section{Problems and results}\label{section1.2}

\subsection{Problems}

Recall that $\mathrm{S}_n$ is the symmetric group on $n$ elements, and $k_j(\sigma)$ is the number of cycles of length $j$ in a permutation $\sigma$. The dissertation focuses on three problems:
\begin{enumerate}
	\item  Asymptotic formulas for the density
\begin{equation}\label{1nu}
 \nu(n,r)=\frac{1}{n!}\big|\{\sigma\in\S_n:\; k_j(\sigma)=0\,\ \forall j\in\overline{r+1, n}\}\big|
\end{equation}
 as $n\to\infty$. The goal is to cover the whole range $1\leq r\leq n$ for the parameter $r$ so that it could be imagined as any function of $n$ falling into the range, thereby $r:\N\to\N$ and $r:=r(n)\leq n$. 
\item We have the same goal for the density
\begin{equation}\label{2kap}
\kappa(n,r)=\frac{1}{n!}\left|\{\sigma\in\S_n :k_j(\sigma)=0\ \forall j\in \overline{1,r}\}\right|.
\end{equation}

\item Let $Z_j$, $j\in\overline{1,n}$, be the Poisson random variables such that $\E Z_j=1/j$. Set $\bar{k}_r(\sigma)=(k_1(\sigma),\ldots,k_r(\sigma))$ and $\bar{Z}_r=(Z_1,\ldots,Z_r)$. The final problem is to obtain an asymptotic formula for the total variation distance between the distributions of random vectors $\bar{k}_r(\sigma)$ and $\bar{Z}_r$. The distance is denoted by $d_{TV}(n,r)$. We have
	\begin{align}
		d_{TV}(n,r)&=\sup_{V\subset\N^r_0}\left|\frac{\#\{\sigma : \bar{k}_r(\sigma)\in V\}}{n!}-\Pr\left(\bar{Z}_r\in V\right)\right|\nonumber\\
		           &=\frac{1}{2}\sum_{m=0}^\infty\nu(m,r)\left|\kappa(n-m,r)-\re^{-H_r}\right|\label{dtvsum}
	\end{align}
	provided that $r:\mathbb{N}\to\mathbb{N}$ and $r=r(n)\leq n$. Thus, the solution to the problem is based on the success in the previous two problems. 
	
\end{enumerate}

	\subsection{Results}
	
	The results for $\eqref{1nu}$ obtained in the thesis are consistently presented in Section~\ref{section2.2}. They improve on all the previous results of such kind posed in Section~\ref{section2.1}, and basically are proved using the saddle-point analysis. Its use is motivated by Chapter~III.5 of \cite{GT} and earlier papers. To be concrete, Theorems~\ref{thm1},~\ref{Theorem 1},~\ref{thm2},~\ref{1cor}, and Corollary~\ref{2cor} are the results on $\nu(n,r)$.
	
	The results for $\eqref{2kap}$ obtained in the work are expounded in Section~\ref{section3.2}. They improve on all the previous results of such type stated in Section~\ref{section3.1}, and basically are proved using the saddle-point analysis, the use of which is motivated by Chapter~III.6 of \cite{GT} and the paper \cite{GT-Crible}. To be specific, Theorems~\ref{thm3.2},~\ref{theorem3.3},~\ref{theorem3.4}, Corollary~\ref{corollary3.5}, and Theorem~\ref{theorem3.2} are the results on $\kappa(n,r)$.
	
	Estimates for \eqref{dtvsum} are given in Theorem~\ref{EMRP} and Corollary~\ref{corollary4.3}.
	Theorem~\ref{EMRP} is an improvement on the previous asymptotic formula for $d_{TV}(n,r)$. We get it by employing the technique developed in \cite{GT-Crible}, which requires obtaining Theorems~\ref{thm2} and~\ref{theorem3.2}. Corollary~\ref{corollary4.3} gives an estimate of \eqref{dtvsum} only through the function $\nu(n,r)$, which could potentially be used to refine previous estimates of $d_{TV}(n,r)$ in the regions where asymptotic formulas are still unknown. 
	
%
%
%
	
\section{Dissemination of the results}

Most of the results can already be found in the papers: 
	\begin{itemize}
\item[$\blacktriangleright$] E.~Manstavi\v cius and R. Petuchovas, \textit{Local probabilities for random permutations without long cycles}, Electron. J. Combin., \textbf{23} (1) (2016), \#P1.58, 25 pp.
\item[$\blacktriangleright$] E.~Manstavi\v cius and R.~Petuchovas, \textit{Local probabilities and total variation distance for random permutations}, The Ramanujan J., (2016), doi: 10.1007/s11139-016-9786-0, 18 pp.
\end{itemize}
The papers are filed in Bibliography as \cite{EMRP-ArX15} and \cite{RAMA2016} respectively. There are only Theorems~\ref{thm3.2},~\ref{theorem3.3},~\ref{theorem3.4}, and Corollary~\ref{corollary3.5} submitted to a journal and waiting for approval.

All the results were exposed and discussed in a series of seminars during the period of 2012-2016 in \underline{Vilnius University} at the \underline{Department of Probability Theory} \underline{and Number Theory} of \underline{Faculty of Mathematics and Informatics}. Also, they were presented at the following conferences: 
\begin{itemize}
\item	R. Petuchovas, Permutations without long cycles, The 55th Conference of the Lithuanian Mathematical Society, 26-27 June 2014.	

\item	R. Petuchovas, On permutations without long cycles, The 11th International Vilnius Conference on Probability Theory and Mathematical statistics, 30 June - 4 July 2014.

\item	R. Petuchovas, The density of the permutations without short cycles in a symmetric group, The 56th Conference of the Lithuanian Mathematical Society, 16-17 June 2015.

\item	R. Petuchovas, Permutations without long or short cycles, European Conference on Combinatorics, Graph Theory and Applications, Norway, Bergen, 31~August - 4~September 2015.
	
\item	R. Petuchovas, Recent results on permutations without short cycles, The 27th International Conference on Probabilistic, Combinatorial and Asymptotic Methods for the Analysis of Algorithms, Poland, Krakow, \mbox{4-8 July 2016.}
\end{itemize}

In addition there are a few conference abstracts worth mentioning:

\begin{itemize}
	\item E. Manstavi\v cius and R. Petuchovas, \textit{Permutations without long or short cycles}, Electronic Notes in Discrete Mathematics, \textbf{49} (2015), 153--158.
	
	\item R.~Petuchovas, \textit{Recent results on permutations without short cycles},
	 \url{http://aofa.tcs.uj.edu.pl/proceedings/aofa2016.pdf} or \arxiv{1605.02690v1}, 2016, 3~pp.
\end{itemize}

The abstracts are filed in Bibliography as \cite{NotesDM2015} and \cite{AofA2016} respectively. 

\newpage

\section{Acknowledgements}

I regard that my grandparents Nina Petuchova and Georgij Petuchov had the strongest influence on my decision to study mathematics. When I was a child, they were the first ones to encourage my love for mathematics, taught and urged me to work hard and properly deal with its problems. I also remember that my mother Rasa Augait\. e was the first one to teach me how to count to one hundred; it happened one dark evening, as we were sitting in the car and waiting for my father finish working... I am certainly grateful to my father Jurij Petuchov for his benevolent and constant support.

I would like to express my gratitude to all my school teachers who taught me mathematics, all of them were very nice; Dalia Ru\v skien\. e from Klaip\. eda former primary school "Au\v srin\. e", and Vida Rudokien\. e, Edmundas Kvederis, Juozas Kukulskis, Vilija \v Sileikien\. e -- all of them from Klaip\. eda high school "\k A\v zuolyno".


I also express my gratitude to the teaching staff of Vilnius University, who lectured me. In addition, I would like to express my gratitude to Antanas Laurin\v cikas, Art\= uras Dubickas, Paulius Drungilas, Ram\= unas Garunk\v stis, and Rima Stan\v cikien\. e, who have had a direct or indirect positive influence on my work at Vilnius University's Department of Probability Theory and Number Theory. 

And, finally, my biggest gratitude for doing mathematics for the past five years goes to a Great Teacher and colleague Professor Eugenijus Manstavi\v cius, who guided me into the subject and who has always been a pleasure to work with. Writing of this thesis would not have been possible without him.

\chapter{Permutations without long cycles}\label{Chapter2} 

\section{Previous results}\label{section2.1}

 Enumeration of decomposable structures missing large components started in 1930 with a paper by K. Dickman \cite{Dickman} dealing with the natural numbers composed of small prime factors. Two decades later, in a series of works, N.G. de Bruijn  extended this to a deep analytical theory. A survey \cite{Moree} gives a broad  historical account. In the 1980's, A. Hildebrand, G. Tenenbaum and other mathematicians again advanced in the same direction. Their results are  well exposed in a book \cite{GT} and in more recent papers. By analogy, a similar theory was carried out for polynomials over a finite field (for example,  \cite{Odlyzko},  \cite{GarePana}) and generalized to the so-called additive arithmetical semigroups (\hspace{-0.02mm}\cite{Warlimont}, \cite{EM-Pal92}, \cite{EM-LMJ92}).  A survey \cite{Granv} discusses the analogy between the theories. The list of the papers does not purport to be complete. However, it has influenced the decision to improve on results of the same problem in case of permutations. 

We focus only on permutations comprising the symmetric group $\S_n$ and seek asymptotic formulas for the probability $\nu(n,r)$
     that a permutation sampled uniformly at random has no cycles of length exceeding  $r$, where  $1\leq r\leq n$, $r\in \N$, and $n\to\infty$.
   The goal is to cover the whole  range for the parameter $r$.

    Let us start from an exact formula. Recall the Notation. If $k_j(\sigma)$ equals the number of cycles of length $j$ in a permutation $\sigma\in\S_n$ and $\bar k(\sigma):=
   \big(k_1(\sigma),\dots,k_n(\sigma)\big)$ is the cycle  structure vector, then 
   \[
                \big|\{\sigma\in\S_n:\; \bar k(\sigma)=\bar s\}\big|={\mathbf 1}\{\ell(\bar s)=n\} n!\prod_{j=1}^n\frac{1}{j^{s_j} s_j!}
   \]
   (\hspace{-0.02mm}\cite[p.~80]{MBona}). Hence,
   \begin{equation}\label{exact1}
                \nu(n,r)=\frac{1}{n!}\big|\{\sigma\in\S_n:\; k_j(\sigma)=0\,\ \forall j\in\overline{r+1, n}\}\big|=
                \sum_{\ell_r(\bar s)=n \atop \ell_{rn}(\bar s)=0}\prod_{j=1}^r\frac{1}{j^{s_j} s_j!},
   \end{equation}
    where the summation is over the vectors $\bar s\in\N_0^n$. The formula can be rewritten in terms of the independent Poisson random variables $Z_j$, $1\leq j\leq n$, such that $\E Z_j=1/j$. To wit,
   \begin{equation}
      \nu(n,r)= \exp\bigg\{\sum_{j^=1}^r\frac{1}{j}\bigg\} \Pr\Big(\sum_{j=1}^rjZ_j=n\Big).
      \label{P-ell}
      \end{equation}
There are  two trivial cases $ \nu(n,1)=1/n!$ and $\nu(n,n)=1$. Also, the first few values of $\nu(n,r)$ are in the \textit{On-Line Encyclopedia of Integer Sequences}:
   \begin{align*}
     \nu(n,2)n!&=A000085,\quad \nu(n,3)n!=A057693,\quad \nu(n,4)n!=A070945,\\
     & \nu(n,5)n!=A070946,\quad \nu(n,6)n!=A070947.
   \end{align*}
   Therein, a great attention is paid to $\nu(n,2)n!$, which equals the number of involutions in $\S_n$. In particular, there is an indication that $\nu(n,2)n!$ can be expressed via a value of the Hermite polynomial $H_n(x)$ at a point on the imaginary axis of a complex plane.

Asymptotic analysis of $\nu(n,r)$ provides more information than the formula~\eqref{exact1} if parameters are large. To start with, we have the Cauchy integral representation
   \begin{equation*}
   \nu(n,r)={1\over 2\pi i}\int_{|z|=\alpha}
  \exp\left\{\sum_{j=1}^r\frac{z^j}{j}\right\} {dz\over z^{n+1}}
   \end{equation*}
   (\eqref{P-ell} and Lemma~\ref{EGF}), where $\alpha>0$. 
    An idea of applying the saddle-point method might be taken from the pioneering work \cite{Pl-Ro29} concerning $H_n(x)$, $x\in\R$. Indeed, this method was used by  L. Moser and M. Wyman \cite{MoWy55} to prove that
     \begin{equation}
     \nu(n,2)\sim \frac{1}{\sqrt{4\pi n}}\exp\Big\{-\frac{n\log n}{2}+\frac{n}{2}+n^{1/2}- \frac{1}{4}\Big\}.\label{nun2}
     \end{equation}
In fact, the relation was established  earlier by S.\ Chowla, I.N.\ Herstein and W.K.\ Moore~\cite{ChHeMo1951}, by another method.
   H.~Wilf included a detailed proof of (\ref{nun2}) into Chapter 5  of his book \cite{Wilf}. However, Exercise 8 on pages 190-191 in it
     gives an erroneous expression for $r=3$. Actually,
         \begin{equation}
    \nu(n,3)\sim\frac{1}{\sqrt{6\pi n}} \exp\Big\{-\frac{n\log n}{3}+\frac{n}{3} +\frac{1}{2} n^{2/3} +\frac{5}{6} n^{1/3}-\frac{5}{18}\Big\}.\label{nun3}
     \end{equation}
    As one has been able to check, the last formula firstly appears in A.N. Timashov's paper \cite{Timash}. It gives a reference to  V.N.~Sachkov's work \cite{Sachkov} where asymptotic formula for $\nu(n,r)$ with $r$ fixed is presented. However, the formula is given without a proof and contains inaccuracies; it  goes in contrast to (\ref{nun3})  and even to (\ref{nun2}). A  year later, M. Lugo \cite{Lugo}  also gave (\ref{nun3}), leaving for the reader other cases of $\nu(n,r)$. 
 
W.K.~Hayman  \cite{Hayman} and later B.~Harris and  L.~Schoenfeld \cite{Har+Schon68} elaborated the methodology that let the latter to succeed in obtaining further asymptotic terms and, in particular, to obtain
\begin{equation} 
\nu(n,r)=\frac{\exp\left\{\sum_{j=1}^r {x^j j^{-1}}\right\}}{x^n\sqrt{2\pi \sum_{j=1}^rjx^j}}\left(1+O_r\Big(\frac{1}{n}\Big)\right)
\label{1Pnr}
\end{equation}
 for an arbitrary but fixed parameter $r\in\overline{1,n}$; where $x$ is the unique positive solution to the equation $\sum_{j=1}^rx^j=n$, and index of $O_r$ stresses indeterminate dependence on the parameter $r$ of constant in the estimate. Actually, we owe to E.~Schmutz whose Theorem 1 and the facts presented below it in  \cite{Schmutz88} clarify  the use of the general and fairly complicated expansion given in \cite{Har+Schon68}. 

 Recently, T.~Amdeberhan and V.H.~Moll \cite[Theorem 8.1]{AmdMoll} have proved by using the asymptotic behavior \eqref{1Pnr} (see also \cite{MoWy1956}) that
\begin{equation}\label{AmMoAsymp}
\nu(n,r)\sim\frac{1}{\sqrt{2\pi nr}}\exp\left\{-\frac{n\log n}{r}+\frac{n}{r}+\sum_{k=1}^{r-1}d_{rk}n^{(r-k)/r}-\frac{1}{r}\sum_{j=2}^r\frac{1}{j}\right\}
\end{equation}
if $r$ is fixed and $n\to\infty$, where $$d_{rk}=\frac{r}{k!(r-k)k}\prod_{j=0}^{k-1}\left(\frac{k+jr}{r}\right).$$ Thus, they have generalized the formulas (\ref{nun2}) and (\ref{nun3}). A proof of the statement is described in the next section in which Theorem~\ref{Theorem 1} validates \eqref{AmMoAsymp} in the region $1\leq r\leq \log n$ and adds an error estimate. 


    The above mentioned results deal with the case when the ratio $n/r$ is "large". Besides that, there is
     the vast extent of literature dealing with the case when $u:=n/r$ is "small".
     Both cases are related to the limit distribution of the longest cycle length, denoted by $L_n(\sigma)$, of $\sigma\in\S_n$.
      V.L.~Goncharov's result  \cite{Gonch44} from 1944 shows that
    \begin{equation}\label{goncharov}
           \nu(n, r)= \frac{1}{n!}\big|\{\sigma\in \S_n:\; L_n(\sigma)\leq r\}\big|=\varrho(u)+o(1)
    \end{equation}
    uniformly in $u\geq 1$. Here  $\varrho(v)$ is the Dickman function, it is continuous for $v>1$. The challenge to estimate the remainder term in \eqref{goncharov} was taken up by X. Gourdon in his notable thesis \cite{Gourd-th}. Formula (12) on page 131 in it gives the result
            \[
           \nu(n, r)=\varrho(u)+O\left({\log n\over r}\right)
    \]
    for $u\geq1$. In addition, an $l$-term asymptotic expansion of $\nu(n,r)$  with a remainder term  $O_l\big( n^{-l+1/2}\log n\big)$ for an arbitrary $l\in\N$ is obtained in Theorem 2 on page 207. Recalling that
  \begin{equation}
  \varrho(v)=\exp\left\{ -v\log v-v\log\log (v+2)+v\left(1+O\left(\frac{\log\log(v+2)}{\log(v+2)}\right)\right) \right\}\label{rho explicit}
  \end{equation}
  if $v\geq 1$ (\hspace{-0.02mm}\cite[Corollary~2.3]{AH+GT}), one can verify that the error terms in either of Gourdon's asymptotic formulas swallow the main term $\varrho(u)$  if $u\geq \log n$. Therefore, it is more desirable to obtain the relative asymptotic estimate. Theorem 4.13 on page 91 in \cite{ABT}, applied for permutations, yields
           \begin{equation*}
           \nu(n, r)=\varrho(u)\big(1+o(1)\big)
    \end{equation*}
    if $u\leq C$, where $C\geq 1$ is an arbitrary but fixed positive constant. As a byproduct of enumeration of corresponding elements in an additive semigroup, the last relation extended to a region $u\leq \sqrt{n/\log n}$ has appeared in the paper \cite{EM-LMJ92}.
     The result is presented in this work as Theorem \ref{thm2} below, it was obtained by E.~Manstavi\v cius and the author \cite{EMRP-ArX15} using another method.

The results exposed in the next section cover what has just been mentioned. 

\section{Results of the thesis}\label{section2.2}

The original source of all the theorems presented in this section is \cite{EMRP-ArX15}. Recall that we are interested in an asymptotic formula for 
\begin{equation*}
                \nu(n,r)=\frac{1}{n!}\big|\{\sigma\in\S_n:\; k_j(\sigma)=0\,\ \forall j\in\overline{r+1, n}\}\big|
   \end{equation*}
when $1< r< n$ and $n\to\infty$. To start with, we have the Cauchy integral representation
   \begin{equation}
   \nu(n,r)={1\over 2\pi i}\int_{|z|=\alpha}
  \exp\left\{\sum_{j=1}^r\frac{z^j}{j}\right\} {dz\over z^{n+1}},
\label{C-integr}
   \end{equation}
   where $\alpha>0$. Applying the classic saddle-point method (described, for example, in~\cite[p.~551]{Fl-Sed}), we obtain 

\begin{theorem}\label{thm1} Let $x$ be the positive solution to the equation $\sum_{j=1}^rx^j=n$.  Then
\begin{equation}
\nu(n,r)=\frac{\exp\left\{\sum_{j=1}^r {x^j j^{-1}}\right\}}{x^n\sqrt{2\pi \sum_{j=1}^rjx^j}}\left(1+O\left(\frac{r}{n}\right)\right)\label{fullx}
\end{equation}
provided that $ 1\leq r\leq c n(\log n)^{-1}(\log\log n)^{-2}$,
where $c=(12\pi^2 \re)^{-1}$ and $n\geq 4$.
\end{theorem}

The upper bound of the parameter $r$ indicates the limitation of the approach. For larger $r$, it is difficult to estimate the integrand in \eqref{C-integr} outside a vicinity of the saddle-point $z=\alpha=x$. The result clarifies formula \eqref{1Pnr} and extends its region of validity. The following theorem is a consequence of Theorem~\ref{thm1}.
   
   \begin{theorem}\label{Theorem 1}  If $ 1\leq r\leq \log n$, then
\[
   \nu(n,r)=
   \frac{1}{\sqrt{2\pi nr}}\exp\bigg\{-\frac{n\log n}{r}+\frac{n}{r}+\sum_{k=1}^r   d_{rk} n^{(r-k)/r}\bigg\} \big(1+O\left( n^{-1/r}\right)\big).
   \]
 Here
\[
   d_{rr}=-\frac{1}{r}\sum_{j=2}^r\frac{1}{j}
   \]
and
\[
             d_{rk}= \frac{\Gamma(k+k/r)}{(r-k)\Gamma(k+1) \Gamma(1+k/r)}
\]
if $1\leq k\leq r-1$.
\end{theorem}

The result improves on \eqref{AmMoAsymp} providing the error estimate and extending the range of the parameter $r$.
   To prove Theorem~\ref{Theorem 1}, we use the Lagrange-B\" urmann formula for the compositions of functions to obtain
\begin{equation*}
 \sum_{j=1}^r\frac{x^j}{j}-n\log x=-\frac{n\log n}{r}+\frac{n}{r}+\sum_{k=1}^r   d_{rk} n^{(r-k)/r}+O\left(n^{-1/r}\right)
\end{equation*}
and
\begin{equation*}
\left(\sum_{j=1}^r jx^j\right)^{-1/2}=\frac{1}{\sqrt{nr}}\left(1+O\left(n^{-1/r}\right)\right)
\end{equation*} 
for $1\leq r\leq \log n$. Clearly, these formulas along with Theorem~\ref{thm1} give the result. Notice, it is not effective here to apply the Lagrange-B\" urmann formula only to the function $x$ and then put the result
\begin{align*}
x&=
n^{1/r}-\frac{1}{r}  -\sum_{l=2}^{r}\frac{ \Gamma(l+(l-1)/r)}{(l-1)\Gamma(l+1)\Gamma((l-1)/r)}
n^{-(l-1)/r}\\
&\quad\   + \frac{1}{r}n^{-1+1/r}+O\left(\frac{1}{n}\right)
\end{align*}
if $2\leq r\leq \log n$ into \eqref{fullx}. This approximation is proved in Lemma~\ref{lema11}.

Recall that  $\varrho$ is the Dickman function. It should be said, $\varrho(v)\leq 1/\Gamma(v+1)$ if $v\geq 1$ \cite[p.~366]{GT}. Denote by $\mathcal{L}(\cdot)$ the Laplace transform. The next theorem is obtained by estimating the difference $\nu(n,r)-\mathcal{L}^{-1}(\mathcal{L}(\varrho)(s))(n/r)$, $s\in\C$. 

\begin{theorem}\label{thm2} Let $u=n/r$. If $\sqrt{n\log n}\leq r\leq n$, then
\[
\nu(n,r)=\varrho(u)\bigg(1+O\left(\frac{u\log(u+1)}{r}\right)\bigg).
\]
\end{theorem}
The result is not new, but the method used to obtain it is different than that in \cite{EM-LMJ92}. It brings some additional understanding how Theorem~\ref{thm2} could be improved.
To do that, one has to use the function $\xi(v)$, which is a positive solution to the equation $\re^{\xi(v)}=1+v\xi(v)$ for $v>1$, where $\xi(1)=0$. The function $\xi(v)$ is important because of its involvement in the asymptotic formula
\begin{equation}\label{varrho1}
\varrho(v)=\sqrt{\frac{\xi'(v)}{2\pi}}\exp\left\{\gamma-v\xi(v)+\int_0^{\xi(v)}\frac{\re^t-1}{t}dt\right\}\Big(1+O\left(\frac{1}{v}\right)\Big)
\end{equation}
valid for $v\geq 1$ (Lemma~\ref{rho}). Formula \eqref{varrho1} gives us an idea that the functions $x$ and $\xi$ are related. Indeed, the analysis carried out in Lemma~\ref{4lema} gives the relations
 \begin{equation*}
x= \exp\Big\{\frac{\xi(u)}{r}\Big\}+O\left(\frac{\log (u+1)}{r^2}\right)
\end{equation*}
and
\begin{equation*}
\Big(\sum_{j=1}^rjx^j\Big)^{-1}= \frac{\xi'(u)}{r^2}\left(1+O\left(\frac{ \log (u+1)}{r}\right)\right)
\end{equation*}
if $\log u\leq r\leq n$. Combining these estimates and Theorem~\ref{thm1}, we get

\begin{theorem} \label{1cor} Let $u=n/r$. If $ n^{1/3}(\log n)^{2/3}\leq r\leq c n(\log n)^{-1}(\log\log n)^{-2}$,
where $c=(12\pi^2 \re)^{-1}$ and $n\geq 4$, then
  \begin{equation}
\nu(n,r)=
\varrho(u) \exp\Big\{\frac{u\xi(u)}{2r}\Big\}\left(1+O\left(\frac{u\log^2(u+1)}{r^2}+\frac{1}{u}\right)\right).
\label{nu-cor1}
\end{equation}
\end{theorem}

The extra exponential factor next to the Dickman function is a significant finding, it is necessary if $r=o\big(\sqrt{n\log n}\big)$ because of $\xi(u)\sim\log u$ when $u\to\infty$ (Lemma~\ref{xi}). In the proof of Theorem~\ref{1cor}, \mbox{we obtain}
\begin{align*}
\frac{\exp\left\{\sum_{j=1}^r {x^j j^{-1}}\right\}}{x^n\sqrt{2\pi \sum_{j=1}^rjx^j}}\left(1+O\left(\frac{1}{u}\right)\right)&=\varrho(u)\left(1+O\left(\frac{u\log(u+1)}{r}\right)\right)
\end{align*}
if $\sqrt{n\log n}\leq r\leq n$. Therefore, we can make the following conclusion.

\begin{corollary}\label{2cor}
Relations $\eqref{fullx}$ and $\eqref{nu-cor1}$ remain to hold if the upper bound of $r$ is substituted by $n$.
\end{corollary}
One can deduce from Corollary~\ref{2cor} that the order of the error estimate in Theorem~\ref{thm2} is precise.

It should be noted that Theorems~\ref{thm2} and~\ref{1cor} are comparable to the similar result of A.~Hildebrand in number theory \cite{H86}.

\chapter{Permutations without short cycles}\label{Chapter3}

\section{Previous results}\label{section3.1}

Recall that $\S_n$ is the symmetric group on a finite set of $n$ symbols, and $k_j(\sigma)$ is the number of cycles of length $j$ in $\sigma$. The task is to find an asymptotic formula for the density
\[
\kappa(n,r)=\frac{1}{n!}\left|\{\sigma\in\S_n :k_j(\sigma)=0\ \forall j\in \overline{1,r}\}\right|
\]
as $1\leq r< n$ and $n\to\infty$. We have the classical example of derangements
\begin{align*}
\kappa(n,1)&=\sum_{j=0}^n\frac{\left(-1\right)^{j}}{j!}=\re^{-1}+O\left(\frac{1}{n!}\right)
\end{align*}
(\hspace{-0.02mm}\cite[p.~135]{1713} or \cite[p.~107]{MBona}) and the trivial case $\kappa(n,r)=1/n$ if $n/2\leq r< n$. Nevertheless, the general formula is complex. Formally, for $1\leq r< n$, we have 
         \[
                \kappa(n,r)=\sum_{\ell_{rn}(\bar s)=n \atop \ell_r(\bar s)=0}\prod_{j=1}^n\frac{1}{j^{s_j} s_j!}
   \]
   (\hspace{-0.02mm}\cite[p.~80]{MBona}), where the summation is over the vectors $\bar s\in\N_0^n$. The formula can be rewritten in terms of the independent Poisson random variables $Z_j$, $1\leq j\leq n$, such that $\E Z_j=1/j$. Namely,
   \begin{equation}
      \kappa(n,r)= \exp\bigg\{\sum_{j=r+1}^{n}\frac{1}{j}\bigg\} \Pr\Big(\sum_{j=r+1}^njZ_j=n\Big).
      \label{P-ell2}
      \end{equation}
On the other hand, the probability $\kappa(n,r)$ is just a mean value of an indicator function $i(\sigma)$ of a relevant subset of permutations. Videlicet,
  \[
       i(\sigma)=\prod_{j=1}^n i_j^{k_j(\sigma)},
       \]
  where $i_j=1$ if $ r<j\leq n$, otherwise $i_j=0$, and $0^0:=1$. General results on mean values of bounded
  multiplicative functions \cite{VZ-RJ09} give a relation
  \[
   \kappa(n,r)=\frac{1}{n!}\sum_{\sigma\in\S_n}i(\sigma)=\re^{-H_r}+O\left(\frac{r}{n}\right)
   \]
which is trivial, if $\sqrt n=o(r)$. To address the results we need the following information. Recall that the Buchstab and Dickman functions, respectively $\omega$ and $\varrho$, are described in Notation. We have $1/2\leq \omega(v)\leq 1$ for $v\geq 1$ \cite[p.~400]{GT}, and $\varrho(v)\leq 1/\Gamma(v+1)$ for $v\geq 1$ \cite[p.~366]{GT}. A function $R(v)$ introduced in \cite{GT-Crible} is described in the thesis by Definition~\ref{Rdefinition} and Lemmas~\ref{omega,R} and~\ref{Rneeded} in Chapter~5. In short, it comes from evaluation of $\omega(v)-\re^{-\gamma}$, and here it is sufficient to know that
\begin{equation}\label{R estimate}
R(v)=\varrho(v)\exp\left\{-\frac{\pi^2v}{2(\log v)^2}+O\left(\frac{v\log\log(v+2)}{(\log v)^3}\right)\right\}
\end{equation}
if $v\geq v_1$, where $v_1$ is some sufficiently large constant, and $R(v)=O\left(1\right)$ if $1\leq v< v_1$. Now, one can mention results obtained so far.

Set $u=n/r$. In 1992 R.~Arratia and S.~Tavar\' e in their notable work~\cite{AT-AP92} estimated the total variation distance for permutations. From the estimate it follows that 
\begin{align}
              \left|\kappa(n,r)-\re^{-H_r}\right|&\leq \sqrt{2\pi \lfloor u \rfloor}\frac{2^{\lfloor u \rfloor-1}}{(\lfloor u \rfloor-1)!}+\frac{1}{\lfloor u \rfloor !}+3\left(\frac{\re}{u}\right)^u\label{tikslus}\\ 
              &=\exp\Big\{-u\log u+u\log (2\re)+O\left(\log (u+2)\right)\Big\}\label{AT}
\end{align}
if $1\leq r< n$ and $H_r=\sum_{j=1}^r1/j$. Later in 2002 E.~Manstavi\v cius~\cite{EM-LMR02} showed that
\begin{equation}\label{EM}
\kappa(n,r)=\re^{-H_r+\gamma}\omega(u)+O\left(\frac{1}{r^2}\right)
\end{equation}
if $1\leq r< n$. His result had not been noticed and weaker results relying on the same proof technique were presented by other authors, to wit, in works \cite{BeMaPaRi-JComTh04} and \cite{Granv-EJC06}. There was applied a recurrence relation
\[
\kappa(n,r)=\frac{1}{n}+\frac{1}{n}\sum_{r< j<n-r}\kappa(j,r),
\]
the induction principle and approximation of sums by integrals. Moreover, all of the latter authors did not notice that there had been a result of  A.~Hildebrand and G.~Tenenbaum
\begin{equation}\label{HT}
\omega(u)=\re^{-\gamma}+O(R(u))
\end{equation}
if $u\geq 1$ (Lemma~\ref{omega,R}), which combined with estimates \eqref{rho explicit}, \eqref{R estimate}, and \eqref{AT} gave
\begin{equation*}
              \kappa(n,r)=\re^{-H_r+\gamma}\omega(u)+O\left(\exp\Big\{-u\log u+u\log (2\re)+O\left(\log u\right)\Big\}\right)
\end{equation*}
for $u\geq 1$. Therefore, in the case $(\log\log n)^{-1}\log n=o(u)$, the result \eqref{tikslus} combined with~\eqref{HT} is stronger than those mentioned after, but in the case $u\leq (\log\log n)^{-1}\log n$ E.~Manstavi\v cius' result (\ref{EM}) provides a better order of the error term. Of course, from the computational point of view, the result \eqref{tikslus} is more serviceable than those with error terms presented using big-O notation.

A recent result by A.~Weingartner~\cite[Proposition 2]{AW-2015} improves the previous ones;
\begin{proposition}\label{AW-proposition} For $1\leq r\leq n/\log n$, we have
\begin{equation}
\kappa(n,r)=\re^{-H_r}+O\left(\frac{(u/\re)^{-u}}{r^2}\right), \label{AW}
\end{equation}
where $u=n/r$. If $r\geq 3$, we can replace $\re$ by 1 in the error term.
\end{proposition}

Also, the paper~\cite{WaMe} by C.~Wang and I.~Mezo has recently appeared on line in which they explore the density of permutations that consists of a particular number of cycles and has no short cycles.
 
The results in the next section cover that presented in this. 

\section{Results of the thesis}\label{section3.2}

To improve on \eqref{AW} and other results presented in the previous section, the author makes use of the Cauchy integral representation 
\[
\kappa(n,r)=\frac{1}{2\pi i}\int_{|z|=\beta}\exp\Big\{\sum_{j=r+1}^n\frac{z^j}{j}\Big\}\frac{dz}{z^{n+1}}
\]
(\eqref{P-ell2} and Lemma~\ref{EGF}), where $\beta>0$ is to be chosen. From the previous results, we know that $\kappa(n,r)\sim \re^{-H_r}$ if $n/r\to\infty$. Therefore, after some transformations of the integrand function, the residue theorem is applied to obtain the main asymptotic term:
\begin{align}
	\kappa(n,r)&=\re^{-H_r}+\frac{1}{2\pi i}\int_{-r\log\alpha-ir\pi}^{-r\log\alpha+ir\pi}f(s)\re^{us}{d}s\label{kappa_e-Hr},
\end{align}
where $u=n/r$,
$
r(\re^{s/r}-1)f(s)=\exp\left\{s/r-\sum_{j=1}^r (\re^{js/r}j)^{-1}\right\},
$
$1\leq r<n$, and $\alpha>1$.
 Then the remaining integral is twice integrated by parts (first with respect to $d\re^{us}$ and then with respect to $d\re^{(u-1)s}$) so that it could be possible to get a more accurate estimate of it:

\begin{theorem}\label{thm3.2}
Let $(t)_+=\max\{t,0\}$, $t\in\R.$
	For all $1\leq r<n/2$ and $\alpha>1$, we have 
	\begin{align*}
		|\kappa(n,r)-\re^{-H_r}| &\leq  \frac{\pi\re^4\alpha^{2r-n+3/2}}{n^2(\alpha-1)^2}\exp\left\{\sum_{j=1}^r\frac{\alpha^j-2}{j}+E(r,\alpha)\right\}\nonumber\\
		&\quad\ +\frac{4\re\alpha^{2r-n+2}}{\pi n^2r(\alpha-1)^3}\exp\left\{-\frac{\alpha(\alpha^r-1)}{2r(\alpha-1)}-H_r\right\}, 	
	\end{align*}
	where $$
	E(r, \alpha)=-\frac{2}{r}\frac{\alpha^{r+1}}{\alpha-1}\left(\frac{\pi^{-2}}{1+(r\alpha-r)^2}-\alpha^{-r/2} \right)_+ +\min\{2r\log\alpha, 2\log(\re r)\}.
	$$
\end{theorem}

We are going to see that appropriate choice of $\alpha$ gives stronger error estimates than \eqref{AT} and \eqref{AW}. However, the author left the reader to find a practical version of Theorem~\ref{thm3.2}, which could be useful in computational mathematics.
         The following three propositions are obtained by virtue of two different approximations of the optimal value of $\alpha:=\alpha(n,r)$. Both approximations satisfy the inequalities $u^{1/r}\leq \alpha\leq u^{2/r}$. 
         Observing that we can exploit the equation \eqref{varrho1}, firstly we take $\alpha=\re^{\xi(u)/r}$:
          
\begin{theorem}\label{theorem3.3}
Let $\varepsilon$ and $\delta$ be arbitrary but fixed positive numbers, and $u=n/r$. We have
\begin{equation*}
\kappa(n,r)=\re^{-H_{r}}+O_{\varepsilon, \delta} \left(\frac{\varrho(u)}{r}\exp\left\{-\frac{2u\left(1-\delta\right)}{\pi^2(\log (1+u))^2}\right\}\right)
\end{equation*} 
if $(\log n)^{3+\varepsilon}\leq r< n$. Moreover, 
\begin{equation*}
\kappa(n,r)=\re^{-H_{r}}+O\left(\frac{\varrho(u)u^{u/r}}{r}\right)
\end{equation*}
if $\log n\leq r< (\log n)^{3+\varepsilon}$. Here $\varrho(u)$ is the Dickman function.
\end{theorem}
Stressing dependence on parameters $\varepsilon$ and $\delta$ of constant in the error estimate, the abbreviation $O_{\varepsilon, \delta}$ with the extra indexes is used. The proof of Theorem~\ref{theorem3.3} can give asymptotic formulas in the whole region $1\leq r< n$, but the precision of them for $2\leq r\leq \log n$ does not satisfy us as much as that in Corollary~\ref{corollary3.5}, which follows from the next theorem. Appealing to Theorem~\ref{thm1} with Corollary~\ref{2cor} and taking $\alpha=x$, we get

\begin{theorem}\label{theorem3.4}
Let $\varepsilon$ be an arbitrary but fixed positive number, and $u=n/r$, then 
\begin{equation*}
\kappa(n,r)=\re^{-H_r}+O_\varepsilon\Bigg(\frac{\nu(n,r)}{r}\exp\left\{-\frac{u^{1-4/r}(1-\varepsilon)}{4\pi^2(\log (u+1))^2}\right\}\Bigg)
\end{equation*}
if $5\leq r< n$, and
\begin{equation*}
\kappa(n,r)=\re^{-H_r}+O\Bigg(\nu(n,r)n^{5/2}\Bigg)
\end{equation*}
if $2\leq r<5$.
\end{theorem}
In the next chapter, it is shown that Theorem~\ref{theorem3.4} can be useful in formulas where both densities $\kappa(n,r)$ and $\nu(n,r)$ are involved. 

Applying Theorem~\ref{Theorem 1} to Theorem~\ref{theorem3.4} we get the following assertion. 

\begin{corollary}\label{corollary3.5} For $2\leq r\leq \log n$, we have
\[
\kappa(n,r)=\re^{-H_r}+O\left(\exp\left\{-\frac{n}{r}\log\frac{n}{\re}+\frac{n}{\log n}+\frac{3n}{(\log n)^2}\right\}\right).
\]
\end{corollary}

The next theorem is also obtained from \eqref{kappa_e-Hr} but relying on other results, which mainly were developed by G.~Tenenbaum to get Theorem~2 of \cite{GT-Crible}.

\begin{theorem} \label{theorem3.2} Let $u=n/r$. If $\sqrt{n\log n}\leq r<n$, then 
	\begin{equation*}
		\kappa(n,r)=\re^{-H_r+\gamma} \omega(u)+O\left(\frac{R(u)u^{3/2}\log^2 (u+1)}{r^2}\right),
	\end{equation*}
	where the function $R(u)$ is discussed in the previous section, see \eqref{R estimate}.
\end{theorem}

This is the result of \cite{RAMA2016}. Comparing to other presented results for the case $\sqrt{n\log n}\leq r<n$, Theorem~\ref{theorem3.2} provides most accurate order of the error term. This theorem is applied in the next chapter and the function $R$ has a vital role there.  

Analytic results obtained for permutations usually can be compared to that obtained in number theory. In our case, one can have in mind analysis of the number of natural numbers missing small prime factors: 
\vspace{-5pt}
\begin{itemize}[itemsep=-3pt]	
\item[-] Theorem~\ref{theorem3.3} is comparable to Corollary 7.4 given in~\cite{GT} on page 417; 
\item[-] Theorem~\ref{theorem3.4} is comparable to Theorem 1 given in~\cite{GT} on page 397; 
\item[-] Theorem~\ref{theorem3.2} is comparable to Corollary~3 of \cite{GT-Crible}. 
\end{itemize}
\vspace{-5pt}
Information about parallelism between theories can be found in~\cite{ABT} or partly in~\cite{AW-2015}.

\chapter{Total variation distance for random permutations}\label{Chapter4}

\section{The state of the art}\label{section4.1}

Recall that $k_j(\sigma)$ is the number of cycles of length $j$ in a permutation $\sigma\in\S_n$. Let $Z_j$, $1\leq j\leq n$, be the random Poisson variables such that $\E Z_j=1/j$. Set $\bar{k}_r(\sigma)=(k_1(\sigma),k_2(\sigma),\ldots,k_r(\sigma))$ and $\bar{Z}_r=(Z_1,Z_2,\ldots,Z_r)$. Taking $\sigma$ uniformly at random from $\S_n$ we observe that $\bar{k}_r(\sigma)$ is a random cycle structure vector of a permutation and its coordinates are random but dependent variables, while the vector $\bar{Z}_r$ coordinates are independent random variables. The task is to estimate the difference between the distributions of the random vectors $\bar{k}_r(\sigma)$ and $\bar{Z}_r$. The distributions we denote by $\mathcal{L}\left(\bar{k}_r(\sigma)\right)$ and $\mathcal{L}\left(\bar{Z}_r\right)$. For that we choose a special metric $d_{TV}$ called a total variation distance. In \mbox{particular case,}
\begin{equation*}
d_{TV}(\mathcal{L}(\bar{k}_r(\sigma)), \mathcal{L}(\bar{Z}_r))=\sup_{V\subset\N^r_0}\left|\frac{\#\{\sigma : \bar{k}_r(\sigma)\in V\}}{n!}-\Pr\left(\bar{Z}_r\in V\right)\right|.
\end{equation*}
Actually,
\begin{align}\label{dtvseries}
d_{TV}(\mathcal{L}(\bar{k}_r(\sigma)), \mathcal{L}(\bar{Z}_r))&=
\frac{1}{2}\sum_{m=0}^\infty \nu(m,r)\left|\kappa(n-m,r) -\re^{-H_r}\right|.
\end{align}
The formula is a simple inference from known facts, for details see Lemma~\ref{dtvformula}. General formulas for $d_{TV}$ are well represented in \cite[p.~67]{ABT}. Further, we denote $d_{TV}(\mathcal{L}(\bar{k}_r(\sigma)), \mathcal{L}(\bar{Z}_r))$ by $d_{TV}(n,r)$ and let 
\[
H(v)=\frac{1}{2}\int_0^\infty\big|\omega(v-t)-{\re}^{-\gamma}\big| \varrho(t) dt +\frac{\varrho(v)}{2}
\]
for $v\geq 1$, where $\varrho$ is the Dickman function and $\omega$ is the Buchstab function
such that $\omega(v)=0$ if $v<1$. It is worth to note that
\begin{equation}
H(v)=\varrho(v)2^{v+v/\xi(v)+O\left(v/\xi(v)^2\right)}
\label{Hu-asymp}
\end{equation}
if $v>1$ \cite[p.~4]{GT-Crible}, where $\xi(v)\sim \log v$ if $v\to\infty$ (Lemma~\ref{xi})  and the order of $\varrho(v)$ is given in Chapter~2 by \eqref{rho explicit}.

In \cite{AT-AP92}, some heuristics was proposed to guess a limit form  of $d_{TV}(n,r)$ as $n\to\infty$. Basing on this, D.~Stark \cite{Stark} established it even for a logarithmic class  of decomposable structures. Confined to permutations, it yields

\begin{proposition}\label{DS}
If $n/r\to\beta\in[1,\infty)$ as $n\to\infty$, then
$ d_{TV}(n,r)\to H(\beta)$.
\end{proposition}

E.~Manstavi\v cius and the author~\cite{RAMA2016} extended the validity of the last relation and obtained a remainder term estimate.

\begin{theorem}\label{EMRP} Let $u=n/r$. If $\sqrt{n\log n}\leq r\leq n$, then
\begin{equation*}
d_{TV}(n,r)=H(u)\bigg(1+O\left(\frac{u^{3/2}\log^2 (u+1)}{r}\right)\bigg).
\end{equation*}
\end{theorem}
To prove Theorem~\ref{EMRP}, we apply Theorems~\ref{thm2} and~\ref{theorem3.2} to \eqref{dtvseries}. The function $R(u)$ in Theorem~\ref{theorem3.2} is crucial, it allows us to get the full asymptotic expansion with $H(u)$ standing for the \mbox{main term.}

It was already indicated in \cite[p.~68]{AT-AP92} that the counts of the first $r$ cycle sizes are well approximated by the independent Poisson random variables, as long as $u\to\infty$. Theorem~\ref{EMRP} shows the rate of the approximation. The author thinks that the error estimate could be refined to $O\left(H(u)u\log (u+1)/r\right)$ at best. The hypothesis follows from the proof of Theorem~\ref{EMRP}. However, it is enough to recall Theorems~\ref{thm2} and~\ref{1cor} with Corollary~\ref{2cor} to understand the reasons. Sure enough, an improvement of the factor $u^{3/2}\log^2(u+1)$ in Theorem~\ref{theorem3.2} would cause an improvement on the same factor in Theorem~\ref{EMRP}.

As it was noted in Chapter~\ref{Chapter3}, R.~Arratia and S.~Tavar\' e~\cite{AT-AP92} showed that  
\begin{align}
	d_{TV}(n,r)&\leq \sqrt{2\pi \lfloor u \rfloor}\frac{2^{\lfloor u \rfloor-1}}{(\lfloor u \rfloor-1)!}+\frac{1}{\lfloor u \rfloor !}+3\left(\frac{\re}{u}\right)^u\nonumber\\ 
	&=\exp\Big\{-u\log u+u\log (2\re)+O\left(\log (u+2)\right)\Big\}\label{estdtv}
\end{align}
if $1\leq r\leq n$. The problem of an upper estimate for $d_{TV}(n,r)$ is not studied in the thesis, but applying Theorem~\ref{theorem3.4} to \eqref{dtvseries} we obtain the hint how \eqref{estdtv} might be improved. We may need only results concerning the density $\nu(n,r)$.
\begin{corollary}\label{corollary4.3}
	If $5\leq r<n$, then
	\[
	d_{TV}(n,r)\ll \frac{1}{r}\sum_{m=0}^{n-r-1} \nu(m,r)\nu(n-m,r)+\frac{1}{r}\sum_{m=n-r}^\infty\nu(m,r)+\nu(n,r).
	\]
\end{corollary}

 \bigskip

\subsection{Historical Note to the proof of Theorem~\ref{EMRP}}

The survey starts with a result called Kubilius' fundamental lemma, after Jonas~Kubilius (\hspace{-0,02mm}\cite{JK1956},~\cite{JK1964}). Denote by $p$ a generic prime number. Let $f$ be an additive arithmetic function and $f_b(n)=\sum_{p\leq b,\ p^v||n}f(p^v)$. Set $Z_b=\sum_{p\leq b}\xi_p$ where $\Pr(\xi_p=f(p^v))=(1-1/p)p^{-v}$ ($v=0,1,2,\ldots$). As noted by Ruzsa and others, the lemma may be stated as follows
\[
d_{TV}(f_b,Z_b):=\sup_{A\subset\R}\left|\eta_a(f_b\in A)-\Pr(Z_b\in A)\right|\ll \re^{-c_1\log a/ \log b}
\]
if $a\geq b\geq 2$ where $\eta_a$ is the uniform measure over the set of natural integers not exceeding $a$, and $c_1$ is some positive constant. Later, M.B.~Barban and A.I.~Vinogradov \cite{Bar-Vin1964} refined Kubilius' estimate, the improvement is described in detail by Elliot~\cite{Elliott}. Eventually, G.~Tenenbaum \cite{GT-Crible} developed a technique, which let him to improve on the results. In particular, he showed 
\[
d_{TV}(f_b,Z_b)=H\left(\frac{\log a}{\log b}\right)\left(1+O\left(\frac{\log\log b}{\log b}\right)\right)
\]
if $a\geq 3$ and $\exp\left\{(\log a)^{2/5+\varepsilon}\right\}\leq b\leq a$ for any fixed $\varepsilon>0$. The technique is implemented in Section~\ref{section5.8}.

\chapter{Proofs}\label{Chapter5}

\section{Auxiliary lemmas}\label{Aux}

Later on the function $\xi(v)$, $v\geq 1$, is defined as in the next lemma. 

\begin{lemma}\label{xi}For $v>1$, define $\xi=\xi(v)$ as the nonzero solution to the equation
\begin{equation*}
                 \re^\xi=1+v\xi
\end{equation*}
and put $\xi(1)=0$. If $v>1$, then $\log v<\xi< 2\log v$,
\begin{equation}\label{xi asymp}
\xi=\log v +\log\log (v+2)+O\left(\frac{\log\log (v+2)}{ \log (v+2)}\right)
\end{equation}
and
\begin{equation}
\xi':=\xi'(v)=\frac{1}{v}\,\frac{\xi}{\xi-1+1/v}=\frac{1}{v}\exp\bigg\{O\left(\frac{1}{\log (v+1)}\right)\bigg\}.
\label{deriv}
\end{equation}
\end{lemma}

\begin{proof}
 To establish the effective bounds for all $v>1$, it suffices to employ the strictly increasing function $(\re^v-1)/v$ if $v\geq 0$. Indeed, the lower bound follows from the inequality
 \[
      v=\frac{\re^{\xi(v)}-1}{\xi(v)}> \frac{\re^{\log v}-1}{\log v}=\frac{v-1}{\log v}
 \]
 following from $ v\log v-v+1>0$  if $v>1$. To prove the upper estimate, it suffices to repeat the same argument.

 The asymptotic formulas for $\xi(v)$ and its derivative can be found in \cite{AH+GT} or in the book \cite{GT}.
\end{proof}

\smallskip

\begin{lemma}\label{EGF}
	Denote a finite set of natural numbers by $I$. Let $Z_j$, $j\in I$, be independent Poisson random variables such that $\E Z_j=1/j$ or, equivalently, $\Pr\left(Z_j=k\right)=\re^{-1/j}/(j^kk!)$, $k\in\N_0$. We have
	\[
	\sum_{n=0}^\infty\Pr\left(\sum_{j\in I}jZ_j=n\right)z^n=\exp\left\{\sum_{j\in I}\frac{z^j-1}{j}\right\},
	\]
	where $z\in\C$.
\end{lemma}

\begin{proof}
	At first, notice that
	\[
	\sum_{n=0}^\infty\Pr\left(\sum_{j\in I}jZ_j=n\right)z^n=\E z^{\sum_{j\in I}jZ_j}.
	\]
	Following calculations are straightforward.
	\begin{align*}
		\E z^{\sum_{j\in I}jZ_j}&=\prod_{j\in I}\E z^{jZ_j}\\
		&=\prod_{j\in I}\sum_{k=0}^\infty\re^{-1/j}\frac{1}{j^kk!}z^{jk}\\
		&=\prod_{j\in I}\re^{-1/j}\sum_{k=0}^\infty\frac{(z^j/j)^k}{k!}\\	
		&=\prod_{j\in I}\re^{(z^j-1)/j}.
	\end{align*}
\end{proof}

\smallskip

\begin{lemma}\label{IT}
Let
\[
I(s)=\int_0^s\frac{\re^t-1}{t}dt
\]
and
\[
T(s)=\int_0^s\frac{\re^t-1}{t}\left(\frac{t}{r}\frac{\re^{\frac{t}{r}}}{\re^{\frac{t}{r}}-1}-1\right)dt,
\]
where $s\in\C$. Then 
\[
I(s)+T(s)=\sum_{j=1}^r\frac{\re^{js/r}-1}{j}.
\]
\end{lemma}

\begin{proof}
Direct integration of the geometric series. 
\end{proof}

\smallskip

\begin{lemma}\label{rho} Let $\varrho(v)$, $v\geq 1$, be the Dickman function. For $v\geq 1$, we have
\[
\varrho(v)=\frac{\exp\big\{\gamma-v\xi(v)+I(\xi(v))\big\}}{\sqrt{2\pi \big(v+(1-v)/\xi(v)\big)}}\Big(1+O\left(\frac{1}{v}\right)\Big).
\]
\end{lemma}

\begin{proof}
See Theorem~8 on p.~374 in \cite{GT}. 
\end{proof}

\smallskip

\begin{lemma} \label{hatrho} Let $$\hat\varrho(s)=\int_0^\infty \re^{-sv} \varrho(v) dv,$$ $s\in\C$, be the Laplace transform of the Dickman function $\varrho(v)$, $v \geq 0$. The function $\hat\varrho(s)$  is entire and
\[
\hat\varrho(s)=\exp\{\gamma+I(-s)\}.
\]
\end{lemma}

\begin{proof}
This is Theorem~7 on p.~372 in \cite{GT}.
\end{proof}

\smallskip

\begin{lemma}\label{rholap} Let
\[
\hat\varrho(s)=\exp\left\{\gamma+I(-s)\right\}, \quad s \in \C,
\]
be the Laplace transform of $\varrho(v)$ (see Lemma~\ref{hatrho}), $s=-\xi(u)+i\tau=:-\xi+i\tau$ and $\tau\in\R$. Then
\begin{equation*}
\hat{\varrho}(s) =\begin{cases} O\left(\exp\left\{I(\xi)-\tau^2u/2\pi^2\right\}\right) & \textrm{if}\quad |\tau|\leq\pi,  \\
O\left(\exp\left\{ I(\xi)-u/(\pi^2+\xi^2) \right\}\right) & \textrm{if}\quad |\tau|>\pi  \end{cases}
\end{equation*}
and
\begin{equation*}
 \hat{\varrho}(s)=\frac{1}{s}\bigg(1+O\left(\frac{1+\xi u}{s}\right)\bigg)\ \textrm{if}\ |\tau|>1+u\xi.
\end{equation*}
\end{lemma}

\begin{proof} 
This is Lemma 8.2 in Section III.5.4 of \cite{GT}.
\end{proof} 

\smallskip

\begin{lemma}\label{zeta} If  $v\geq v_1$, where $v_1>1$ is sufficiently large constant, then
there exists a unique complex solution $\zeta=\zeta_0(v)$ to the equation
 \[
                 \re^\zeta=1-v\zeta
 \]
satisfying
\[
        \big|\zeta_0(v)-\xi(v)+i\pi\big|\leq \pi.
        \]
 Moreover,
\[
\zeta_0(v)=\xi(v)+\frac{\pi^2}{2\xi(v)^2}-i \frac{\pi\xi(v)}{\xi(v)-1}+O\left(\frac{1}{\xi(v)^3}\right).
\]
\end{lemma}

\begin{proof}
See \cite[p.~6]{GT-Crible}.
\end{proof} 

\smallskip

\begin{lemma} \label{IsJs} Let $ s\in\C\setminus]-\infty,0]$,
\[
     J(s):=\int_0^\infty\frac{\re^{-s-t}}{s+t} dt.
     \]
Then
     \[
              I(-s)+J(s)+\gamma+\log s=0
     \]
and, if $\tau:=\Im s\not=0$,
\[
   J(s)=\frac{\re^{-s}}{s}\Big(1-\frac{1}{s}\Big)+ O\left(\frac{\re^{-s}}{|\tau|^3}\right).
   \]
   \end{lemma}

   \begin{proof}
   These are Lemma 7.1 on p.~371  of \cite{GT} and a formula on p.~406 of the same book, respectively.
   \end{proof}    

\smallskip

\begin{lemma} \label{hatomega} The Laplace transform
\begin{equation}
   \hat\omega(s)=\int_0^\infty \re^{-st} \omega(t) dt, \quad \Re s>0,
   \label{hatomeg}
   \end{equation}
    extends to a meromorphic function on $\C$ by the formula
\begin{equation}
   1+ \hat\omega(s)=\frac{1}{s\hat\varrho(s)}, \quad s\not=0.
\label{mero}
\end{equation}
Moreover,
\[
   1+\hat\omega(s)=\exp\{J(s)\}
   \]
 if $s\in\C\setminus]-\infty,0]$ and
\begin{equation}
s\hat\omega'(s)=-\re^{-s}\big(\hat\omega(s)+1\big)
\label{tapat}
\end{equation}
if $s\not=0$.
\end{lemma}

\begin{proof}
The first assertion is just Theorem~5 given on p.~404 of~\cite{GT}. The last relation follows from (\ref{mero}) and Lemmas~\ref{hatrho} and~\ref{IsJs}.
\end{proof} 

\smallskip

\begin{lemma}\label{Robert} Let $u=n/r$ and $z=-\xi_0(u)+i\tau$, where 
$$\xi_0(u):=\left\{
\begin{array}{ll}
\Re\zeta_0(u) & (1< v_1\leq u)\\
\xi_0(v_1)    & (1\leq u\leq v_1)
\end{array}
\right.$$
and $\zeta_0$ is defined in Lemma~\ref{zeta}. Uniformly in  $(\log n)^2\leq r\leq n/2$ and $|\tau|\leq \re^{\sqrt r}$, we have
\begin{equation*}
e^{uz}\left(1+\hat{\omega}(z)\right)\ll R(u)\sqrt{u},
\end{equation*}
where the function $R$ is defined by Definition~\ref{Rdefinition} and Lemma~\ref{omega,R} below.
\end{lemma}

\begin{proof}
At first, an excursion  to number theory is worthy. Let $\zeta(s)$ be the Riemann zeta function defined for $\Re s>1$ by
\[
        \zeta(s)=\sum_{m=1}^\infty \frac{1}{m^s}=\prod_{p}\Big(1-\frac{1}{p^s}\Big)^{-1}
  \]
and extended analytically to the whole complex plane with the exception of $s=1$. Here $p$ denotes a prime number. Set
\[
        \zeta(s, y)=\prod_{p\leq y}\Big(1-\frac{1}{p^s}\Big)^{-1}, \qquad  L_{\varepsilon}(y)=\exp\big\{(\log y)^{3/5-\varepsilon}\big\},\quad y\geq 2.
        \]
Lemma 9.1 in \cite[p.~378]{GT}  asserts that given $\varepsilon>0$ there exists a $y_0=y_0(\varepsilon)$ such that, under conditions
\[
y\geq y_0, \qquad \Re s\geq1-(\log y)^{-(2/5)-\varepsilon}, \qquad |\Im s|\leq L_{\varepsilon}(y),
\]
we have uniformly
\begin{equation}
   \zeta(s,y)=\zeta(s)(s-1) (\log y)\hat\varrho((s-1)\log y)\Big(1+O\left(\frac{1}{L_\varepsilon(y)}\right)\Big).
\label{zet1}
\end{equation}
On the other hand, it has been proved in \cite[Lemma 9]{GT-Crible} that, for $\exp\big\{ (\log\log x)^2\big\}\leq y\leq x$, $x\geq 3$,
\begin{equation}
\frac{x^{s-1}\zeta(s)}{\zeta(s,y)}\ll R\Big(\frac{\log x}{\log y}\Big)\sqrt{\frac{\log x}{\log y}}
\label{zet2}
\end{equation}
on the line segment
\[
\Re s=1-\xi_0\Big(\frac{\log x}{\log y}\Big)\frac{1}{\log y}, \qquad |\Im s|\leq L_\varepsilon(y).
\]

Thereby, the argument is the following. Let $n$ be  sufficiently large. Taking into account Lemma~\ref{hatomega}, it suffices to compare relations (\ref{zet1}) and (\ref{zet2}) with $\varepsilon=1/15$, $x=\re^n$, $y=\re^r$ and $z=(s-1)\log y$. 
\end{proof} 

\smallskip

\begin{definition}\label{Rdefinition} 
	Recall the function $\zeta_0$ of Lemma~\ref{zeta}. If $v\geq v_1$, where $v_1>1$ is sufficiently large constant, then 
	\[
	R(v)=\left|\frac{\exp\left\{-v\zeta_0(v)-I(\zeta_0(v))\right\}}{\zeta_0(v)\sqrt{2\pi v(1-1/\zeta_0(v))}}\right|.
	\]
	If $1\leq v<v_1$, we set $R(v)=O(1).$ See \cite[p.~6]{GT-Crible}.
\end{definition}

\smallskip

\begin{lemma}\label{omega,R}
	Let $\omega(v)$, $v\geq 1$, be the Buchstab function. Then, for $v\geq 1$, we have
	$$
	\omega(v)-{\re}^{-\gamma}=-2{\re}^{-\gamma}R(v)\left(\cos\vartheta(v)+O\left(1/v\right)\right),
	$$
	where $R(v)$ and $\vartheta(v)$ are real valued differentiable functions. Moreover, $R(v)$ is decreasing for sufficiently large $v$ and 
	\[
	R(v)=\varrho(v)\exp\left\{\frac{-\pi^2v}{2\xi(v)^2}+O\left(\frac{v}{\xi(v)^3}\right)\right\}.
	\]  
\end{lemma}

\begin{proof}
	This is Lemma 4 in~\cite{GT-Crible}.
\end{proof}

\smallskip

\begin{lemma}\label{Rneeded} For a sufficiently large constant $v_1$, we have 
\[
R(v)=\varrho(v)\exp\left\{\frac{-\pi^2v}{2(\log v)^2}+O\left(\frac{v\log\log(v+2)}{(\log v)^3}\right)\right\}
\] 
if $v\geq v_1$ and $R(v)=O(1)$ if $1\leq v<v_1$.
\end{lemma}

\begin{proof}
This is a conclusion of Lemma~\ref{xi} and Lemma~\ref{omega,R}.
\end{proof}

\smallskip

\begin{lemma}\label{x} Let $1\leq r\leq n$. Denote $x$ to be the positive solution to the equation 
 $\sum_{j=1}^rx^j=n$. Set $u=n/r$ and $\lambda(x)=\sum_{j=1}^rjx^j$. We have
\[
\re^{(\log u)/r}\leq x \leq \re^{(2\log u)/r}
\]
if $u\geq 1$,
\begin{equation}
          x=\exp\bigg\{{\log\big( u\cdot\min\{r, \log u\}\big)\over r}\bigg\}\bigg(1+O\left({1\over r}\right)\bigg)
\label{wedge}
\end{equation}
if $u\geq 3$, and
 \begin{align}
x&=\exp\left\{\frac{\log\left(u\log u\right)}{r}\right\}\left(1+O\left(\frac{\log\log u}{r\log u}\right)+O\left(\frac{\log u}{r^2}\right)\right)
\label{xxi}
\end{align}
if $3\leq u\leq \re^r$. Moreover, for $u>1$,
\begin{equation}
   |\lambda(x)/(r^2u)-1|\leq \log^{-1} u.
\label{lambda}
\end{equation}
\end{lemma}

\begin{proof}  By definition,  $x>1$ and  $u\leq x^r\leq ru$ for $u>1$. The well-known property of geometric
   and arithmetic means
\[
    x^{(r+1)/2}=(x^1 x^2\cdots x^r)^{1/r}\leq {1\over r}\sum_{j=1}^r x^j=u
    \]
yields
\begin{equation}
              u^{1/r}\leq x\leq u^{2/(r+1)}\leq u.
\label{urx}
\end{equation}
We have from the definition that
  \begin{equation}
   x^r=1+ru(1- x^{-1}).
\label{xr}
\end{equation}
Consequently, by (\ref{urx}) and by virtue of $1-\re^{-t}\geq t\re^{-t}$ if $t\geq 0$,
\[
   x^r> ru\big(1-\exp\{-(\log u)/r\}\big)\geq u\log u \exp\{-(\log u)/r\}\geq {\re}^{-1} u\log u
\]
 provided that $r\geq \log u$. Similarly,
  \[
   x^r\leq 1+ru(1-\exp\{-2(\log u)/r\})\leq 1+2u\log u.
   \]
   The last two inequalities imply
   \begin{equation}
            r\log x =\log (u\log u)+ O\left(1\right)
   \label{rlog}
   \end{equation}
   for $r\geq \log u$. If $r\leq  \log u$, we  have
 \[
   x^r> ru\big(1-\exp\{-(\log u)/r\}\big)\geq \big(1-\re^{-1}\big) ru.
\]
and $x^r\leq 1+ru$. Now,
\[
            r\log x =\log (ur)+ O\left(1\right).
   \]
      The latter and (\ref{rlog}) lead to relation (\ref{wedge}).

 To sharpen (\ref{wedge}) for $3\leq u\leq \re^r$, we  iterate once more and obtain
\begin{align*}
  r\log x &=\log \Big[1+ru\big(1-x^{-1}\big)\Big]\\
   &=
     \log\bigg[1+ r  u\bigg(1-\exp\bigg\{{-\log( u\log u)\over r}\bigg\}\Big(1+O\left(1/r\right)\Big)\bigg)\bigg]\\
     &=
         \log\Big( u\log( u\log u)+O\left(u\right)+O\left(u\log^2 u/r\right)\Big)\\
     &=
     \log( u\log u)+O\left({\log\log u\over \log u}\right)+O\left({\log u\over r}\right).
         \end{align*}
 This is (\ref{xxi}). To prove (\ref{lambda}), we first of all observe that
\begin{equation}
\lambda(x)=\frac{rx^{r+1}}{x-1}-\frac{x^{r+1}-x}{(x-1)^2}=\frac{rx^{r+1}-ru}{x-1}= r^2u+\frac{r(x-u)}{x-1}.
\label{Lambda}
\end{equation}
Further,
\[
0\leq \frac{1}{ru} \frac{u-x}{x-1}<\frac{1}{r(x-1)}\leq \frac{1}{\log u},
\]
due to  (\ref{urx}) and  $r(x-1)\geq r(\re^{(\log u)/r}-1)\geq \log u$.
\end{proof}

\smallskip

\begin{lemma}\label{very small} For $1\leq r\leq \log n$, we have
\[
\nu(n,r)\ll \exp\left\{-\frac{n\log n}{r}+\frac{n}{r}+\frac{n}{\log n}+\frac{3n}{(\log n)^2}-\frac{2n}{(\log n)^3}+\log\frac{\log n}{r}\right\}.
\]
\end{lemma}

\begin{proof}
We apply Theorem~\ref{Theorem 1}, 
\begin{align*}
\nu(n,r)&= \frac{1}{\sqrt{2\pi nr}}\exp\left\{-\frac{n\log n}{r}+\frac{n}{r}+\sum_{N=1}^rd_{rN}n^{(r-N)/r}\right\}\left(1+O\left(n^{-1/r}\right)\right)\\
&\ll  \exp\left\{-\frac{n\log n}{r}+\frac{n}{r}+\sum_{N=1}^rd_{rN}n^{(r-N)/r}-\log r \right\},
\end{align*}
where $d_{rr}=-(1/r)\sum_{j=2}^r1/j$ and
\[
d_{rN}=\frac{\Gamma(N+N/r)}{(r-N)\Gamma(N+1)\Gamma(1+N/r)}\leq \frac{1}{r-N}
\]
if $0 < N/r < 1 $. Therefore,
\begin{align*}
\nu(n,r) &\ll \exp\left\{-\frac{n\log n}{r}+\frac{n}{r}+\sum_{N=1}^{r-1}\frac{n^{N/r}}{N}-\log r \right\}
\end{align*}
and the proposition follows from an estimate
\begin{align*}
\sum_{N=1}^{r-1}\frac{n^{N/r}}{N}&\leq  \int_{1}^{\log n}\frac{{\re}^t-1}{t}dt+\log\log n\\
&=  \frac{{\re}^t-t-1}{t}\Big|_1^{\log n}+\frac{{\re}^t-t^2/2-t-1}{t^2}\Big|_1^{\log n}+2\int_{1}^{\log n}\frac{{\re}^t-t^2/2-t-1}{t^3}dt\\
&\quad\ +\log\log n\\
&\leq  \frac{n}{\log n}+\frac{n}{(\log n)^2}+2(\log n-1)\frac{n}{(\log n)^3}+\log\log n.
\end{align*}
\end{proof}

\smallskip

 \begin{lemma}\label{Tz} Let
\[
T(z)=\int_0^{z}\frac{\re^t-1}{t}\left(\frac{t}{r}\frac{\re^{t/r}}{\re^{t/r}-1}-1\right) dt, \quad z\in \C.
\]
If $z=\eta+i\tau$, $0\leq \eta\leq \pi r$ and $|\tau|\leq \pi r$, then
\begin{equation}
\Big|T(z)+\frac{z}{2r}\Big|\leq \frac{4\re^\eta}{r}+\frac{\tau^2}{12r^2},
\label{Tz1}
\end{equation}

\begin{equation}
\Big| T(\eta)-\frac{1}{2r}(\re^{\eta}-\eta-1)\Big|\leq \frac{\eta\re^{\eta}}{4r^2},
\label{Tz2}
\end{equation}

\begin{equation}\label{Tz3}
T(\eta)\leq \frac{\re^\eta}{2r}+\frac{\eta\re^\eta}{12r^2}.
\end{equation}
\end{lemma}

\begin{proof} The well known theory of the Bernoulli numbers $\{b_n\}$, $n\geq 0$, gives us the series
\begin{equation}
 b(w):=    \frac{w}{1-\re^{-w}}=\sum_{n=0}^\infty \frac{b_n(-w)^n}{n!}=1+\frac{w}{2}+2\sum_{k=1}^\infty \frac{(-1)^{k+1}\zeta(2k)}{(2\pi)^{2k}} w^{2k}\label{Bern}
 \end{equation}
 converging for $|w|<2\pi$, $w\in\C$. Here $\zeta(2k)=\sum_{m\geq 1}m^{-2k}\leq \zeta(2)=\pi^2/6$. Hence,
\begin{align}
T(z)&=\frac{1}{2r}\int_0^{z} \left(\re^{t}-1\right) dt+2\sum_{k=1}^\infty \frac{(-1)^{k+1}\zeta(2k)}{(2\pi r)^{2k}}\int_0^{z}(\re^t-1)t^{2k-1}dt\nonumber\\
&=\frac{1}{2r}(\re^{z}-z-1)+2\sum_{k=1}^\infty \frac{(-1)^{k+1}\zeta(2k)}{(2\pi r)^{2k}}\left( \re^{z}z^{2k-1}-(2k-1)\int_0^{z}\re^t t^{2k-2}dt\right)
\nonumber\\
&\quad\ +2\sum_{k=1}^\infty \frac{(-1)^{k}\zeta(2k)z^{2k}}{2k(2\pi r)^{2k}}.
\label{Ts}
\end{align}
Under the assumed conditions, $|z|^2\leq 2\pi^2 r^2$; therefore, summing up the series, we obtain
\begin{align*}
\Big|T(z)+\frac{z}{2r}\Big|&\leq
\frac{\re^\eta}{r}+\frac{2\pi^2}{3}\re^\eta\sum_{k=1}^\infty\frac{|z|^{2k-1}}{(2\pi r)^{2k}}+\frac{\pi^2}{6}\sum_{k=1}^\infty\frac{|z|^{2k}}{k(2\pi r)^{2k}}\nonumber\\
&\leq
\frac{\re^\eta}{r}+\frac{\re^{\eta}(\eta+|\tau|)}{3 r^2}+\frac{\eta^2+\tau^2}{12r^2}
\leq \frac{\re^\eta}{r}\Big(1+\frac{2\pi}{3}+\frac{\pi}{12}\Big)+\frac{\tau^2}{12r^2}\\
&\leq
 \frac{4\re^\eta}{r}+\frac{\tau^2}{12r^2}.
\end{align*}
To prove (\ref{Tz2}), it suffices to repeat  estimation of the two series in  (\ref{Ts}).
To obtain (\ref{Tz3}), observe that
\[
\frac{\eta\re^\eta}{\re^\eta-1}\leq 1+\frac{\eta}{2}+\frac{\eta^2}{12}
\]
and thus,
\[
T(\eta)\leq \frac{1}{2r}\int_0^\eta(\re^t-1)dt+\frac{1}{12r^2}\int_0^\eta t(\re^t-1)dt.
\]
\end{proof}

\medskip

\section{Proof of Theorem~\ref{thm1}}

\medskip

\textbf{Theorem~\ref{thm1}.} \textit{Let $x$ be the positive solution to the equation $\sum_{j=1}^rx^j=n$.  Then
\begin{equation*}
\nu(n,r)=\frac{\exp\left\{\sum_{j=1}^r {x^j j^{-1}}\right\}}{x^n\sqrt{2\pi \sum_{j=1}^rjx^j}}\left(1+O\left(\frac{r}{n}\right)\right)
\end{equation*}
provided that $ 1\leq r\leq c n(\log n)^{-1}(\log\log n)^{-2}$,
where $c=1/(12\pi^2 \re)$ and $n\geq 4$.\\}

 The essential part of the proof concerns the following trigonometric sum
\[
g_r(t, y):=\sum_{j\leq r}\frac{y^j(\re^{itj}-1)}{j},\quad t\in (-\pi, \pi], \; y>1.
\]
 Its behavior  outside a vicinity of the point $t=0$ is rather complicated; therefore, we consider it in a separate lemma.
Denote
\[
\lambda_k:=\sum_{j=1}^rj^{k-1}x^j,\ k\geq 1.
\]
In particular, $\lambda_1=n$.

\begin{lemma}
\label{7lema} If $t\in [-\pi,\pi]$ and $y>1$, then
\begin{equation}
\Re g_r(t,y)\leq -\frac{2}{\pi^2}\frac{y^{r+1}}{r(y-1)}\frac{t^2}{(y-1)^2+t^2}+\frac{2y}{r(y-1)}.
\label{grt}
\end{equation}

If  $1/r\leq |t|\leq \pi$, $u=n/r\geq 3$, and $x$ is function defined in Theorem~\ref{thm1}, then
\begin{equation*}
\Re g_r(t,x)\leq -\frac{1}{4\pi^2} \frac{u^{1-4/(r+1)}}{\log^2 u}+\frac{2}{r}+\frac{2}{\log u}.
\end{equation*}
 \end{lemma}

\begin{proof} Observe that
\begin{align}
\Re\sum_{j=1}^r\frac{y^j(e^{itj}-1)}{j} &\leq  \frac{1}{r} \Re \sum_{j=1}^ry^j(e^{itj}-1)\nonumber\\
&= \frac{y^{r+1}}{r(y-1)}\left(\Re \frac{e^{it(r+1)}(y-1)}{ye^{it}-1}-1\right)+\frac{y}{r(y-1)}\left(1-\Re \frac{e^{it}(y-1)}{ye^{it}-1}\right)\nonumber\\
&\leq
\frac{y^{r+1}}{r(y-1)}\left(\frac{y-1}{|ye^{it}-1|}-1\right)+\frac{2y}{r(y-1)}.
\label{grt0}
\end{align}
If $|t|\leq \pi$, then
\[
\frac{|ye^{it}-1|}{y-1}=\left(1+\frac{2y(1-\cos{t})}{(y-1)^2}\right)^{\frac{1}{2}}\geq\frac{((y-1)^2+(4/\pi^2)t^2)^{\frac{1}{2}}}{y-1}
\]
because of
\begin{equation}
2t^2/\pi^2\leq 1-\cos{t}\leq t^2/2.
\label{cos1}
\end{equation}
Using also
\[
\frac{\alpha}{\sqrt{\alpha^2+v^2}}-1\leq -\frac{1}{2}\frac{v^2}{\alpha^2+v^2},\quad  \alpha\geq 0,\; v\in\R,
\]
with $\alpha=y-1$ and $v=(2/\pi) t$,
we obtain
\[
\frac{y-1}{|ye^{it}-1|}-1\leq -\frac{2}{\pi^2}\frac{t^2}{(y-1)^2+t^2}.
\]
Inserting this into (\ref{grt0}),  we complete the proof of inequality (\ref{grt}).

If $y=x$, $1/r\leq |t|\leq \pi$ and $u\geq 3$, we  combine (\ref{grt}) with estimate (\ref{urx}). We have
\[
    \frac{x^{r+1}}{x-1}=n+\frac{x}{x-1}\geq ur
    \]
and
\[
  1<\log u\leq r(x-1)\leq r(u^{2/(r+1)}-1)\leq \frac{2r}{r+1} u^{2/(r+1)}\log u.
  \]

   So,  we obtain
\begin{align*}
\Re g_r(t,x)&\leq -\frac{1}{\pi^2} \frac{u}{r^2(x-1)^2}+\frac{2}{r}\Big(1+\frac{1}{x-1}\Big)\\
&\leq
-\Big(\frac{r+1}{2\pi r}\Big)^2 \frac{u^{1-4/(r+1)}}{\log^2 u}+\frac{2}{r}+\frac{2}{\log u}.
\end{align*}.
\end{proof}

\begin{proof}[Proof of Theorem~\ref{thm1}]
Recall \eqref{1Pnr}, so it suffices to consider the case when $r\geq 4$ and $n$ gets large. It is more convenient to examine the probability $\Pr\Big(\sum_{j=1}^rjZ_j=n\Big)$ introduced in (\ref{P-ell}). Set 
 \begin{equation}
 Q(z)= z^{-n} \exp\left\{\sum_{j=1}^r\frac{z^j-1}{j}\right\}.
 \label{Qz-naujas}
 \end{equation}
In the  introduced notation, we have $u\geq c^{-1} (\log n)(\log\log n)^2$ and
\begin{align}
\Pr\Big(\sum_{j=1}^rjZ_j=n\Big)&= \frac{Q(x)}{2\pi}\left(\int_{|t|\leq t_0}+\int_{t_0<|t|\leq \pi}\right)\exp\left\{g_r(t,x)\right\}e^{-itn}dt\nonumber\\
&=:
\frac{Q(x)}{2\pi}\big(K_1(n)+K_2(n)\big)\label{P-int}
\end{align}
with  $t_0:=r^{-7/12}n^{-5/12}$.

Expanding the integrand in $K_1(n)$, we use relations
${\re}^{it}=1+it-t^2/2-it^3/6+O\left(t^4\right)$ if  $ t\in\R$ and
$    \re^{w}=1+O\left(|w| \re^{|w|}\right)$ if $ w\in\C$.
 Consequently, checking that $\lambda_4t_0^4\leq (r^3 n)(r^{-7/3}n^{-5/3})= (r/n)^{2/3}\leq 1$ and using the abbreviation $\lambda:=\lambda_2$,  we obtain
\begin{align*}
&\exp\{g_r(t,x)\}\\
&=\exp\big\{i \lambda_1t-(\lambda/2)t^2-i(\lambda_3/6)t^3+O\left(\lambda_4t^4\right)\big\}\\
&=\exp\big\{it \lambda_1-(\lambda/2)t^2\big\}\big(1-i(\lambda_3/6)t^3 +O\left(\lambda_3^2 t^6\right)\big)+O\left(\lambda_4 t^4 \exp\big\{-(\lambda/2)t^2\big\}\right)\\
&=\exp\big\{it \lambda_1-(\lambda/2)t^2\big\}\big(1-i(\lambda_3/6)t^3\big)+O\left(\big(\lambda_4t^4+\lambda_3^2 t^6\big)\exp\big\{-(\lambda/2)t^2\big\}\right).
\end{align*}
Recall that $u=n/r$, $\lambda_1=n$, $\lambda_k\leq r^k u$ if $k\geq 1$, and,  by Lemma \ref{x}, $\lambda=\lambda(x)\sim nr$  as $n\to\infty$  because of $u\to\infty$. We now see that
\begin{align*}
K_1(n)&=\int_{|t|\leq t_0} \re^{-(\lambda/2)t^2} dt +O\left(\frac{1}{\sqrt\lambda}\left(\frac{\lambda_4}{\lambda^2}+\frac{\lambda_3^2}{\lambda^3}\right)\right)\\
&=
\sqrt{\frac{2\pi}{ \lambda}}-\frac{1}{\sqrt{\lambda}}\int_{|v|> t_0\sqrt{\lambda}} \re^{-v^2/2} dv+O\left(\frac{1}{u\sqrt{\lambda}}\right)=
\sqrt{\frac{2\pi}{ \lambda}}+O\left(\frac{1}{u\sqrt{\lambda}}\right).
\end{align*}

Considering  $K_2(n)$, we firstly observe that, by virtue of (\ref{cos1}),
$
\Re g_r(t,x)\leq -(2/\pi^2) \lambda t^2
$
if $t_0\leq |t|\leq 1/r$. Therefore, the contribution of the integral over this interval to $K_2(n)$ equals  $O\left((u\sqrt{\lambda})^{-1}\right)$.

Further, we apply Lemmas \ref{x} and \ref{7lema} to get
  \begin{align*}
  K_2(n)&\ll \max_{1/r\leq |t|\leq \pi} \big|\exp\big\{g_r(t,x)\big\}\big| +\frac{1}{u\sqrt{\lambda}}\\
  &\ll \frac{1}{\sqrt\lambda} \exp\bigg\{-\frac{1}{4\pi^2} \frac{u^{1-4/(r+1)}}{\log^2 u}+\frac{1}{2}\log u+\log r\bigg\}+\frac{1}{u\sqrt{\lambda}}.
  \end{align*}
It remains to prove that the quantity in the large curly braces does not exceed $-\log u+O\left(1\right)$ if the bounds of $r$ are as in Theorem \ref{thm1}. This is trivial, if
$4\log u>r+1\geq 5$.
If $4\log u\leq r+1$ and $n$ is sufficiently large, we have an estimate
\[
   \frac{1}{4\pi^2} \frac{u^{1-4/(r+1)}}{\log^2 u}\geq  \frac{3 c u}{\log^2 u}\geq  \frac{3\log n(\log\log n)^2}
      {(\log\log n+2\log\log\log n +O\left(1\right))^2}\sim 3\log n
\]
which assures the desired  bound $K_2(n)=O\left((u\sqrt \lambda)^{-1}\right)$.

   Inserting  the estimates of $K_j(n)$, $j=1,2$, into (\ref{P-int}), we finish the proof of Theorem~\ref{thm1}.
\end{proof}

\medskip

\section{Proof of Theorem~\ref{Theorem 1}}\label{s:5}

\medskip

\textbf{Theorem~\ref{Theorem 1}.} \textit{If $ 1\leq r\leq \log n$, then
\[
   \nu(n,r)=
   \frac{1}{\sqrt{2\pi nr}}\exp\bigg\{-\frac{n\log n}{r}+\frac{n}{r}+\sum_{N=1}^r   d_{rN} n^{(r-N)/r}\bigg\} \big(1+O\left( n^{-1/r}\right)\big).
   \]
 Here
\[
   d_{r,r}=-\frac{1}{r}\sum_{j=2}^r\frac{1}{j}
   \]
and
\[
             d_{rN}= \frac{\Gamma(N+N/r)}{(r-N)\Gamma(N+1) \Gamma(1+N/r)}
\]
if $1\leq N\leq r-1$.}

Primarily, the lemmas needed for the proof of Theorem~\ref{Theorem 1} will be presented.
Let $\C[[w]]$ be the set of formal power series over the field $\C$ and let $[w^n] g(w)$ denote the $n$th coefficient of $g(w)\in \C[[w]]$ where $n\in\N_0$.

\begin{lemma} \label{lema1}
Let
\[
\Phi(w)=\sum_{N=0}^\infty\Phi_N w^N
\]
be a power series in $\C[[w]]$ with $\Phi_0=1$. Then, the equation $w=z\Phi(w)$ admits a
unique solution
 \[
w=f(z)=\sum_{N=1}^\infty f_Nz^N, \quad   f_N=\frac{1}{N} [w^{N-1}]\Phi(w)^N, \quad N\geq 1.
\]
\end{lemma}

\begin{proof}
This is Lagrange-B\"{u}rmann Inversion Formula, presented, for instance on page 732 of a fairly concise book \cite{Fl-Sed}.
\end{proof} 

\smallskip

In this section, superpositions of series involving $f(z)$ are used, therefore we recall more variants of the inversion formula. Let us stress that, by Lemma \ref{lema1}, $f_1=1$; therefore,  $z/f(z)$ and $\log \big(z/f(z)\big)$ have  formal power series expansions.

\begin{lemma} \label{lema2} Let $f(z)$ be as in Lemma $\ref{lema1}$ and  $j\in \N$. Then
\[
   [z^N]\Big(\frac{z}{f(z)}\Big)^j= \frac{j}{j-N}[w^N]\Phi(w)^{N-j}
   \]
 if $N\in \N_0\setminus\{j\}$ and
\[
            [z^j]\Big(\frac{z}{f(z)}\Big)^j=-[w^{j-1}]\Big(\frac{\Phi'}{\Phi}(w)\Big).
                        \]
Moreover,
\[
   [z^N]\log \frac{z}{f(z)}= -\frac{1}{N}[w^N]\Phi(w)^{N}
   \]
   if $N\geq1$.
\end{lemma}

\begin{proof} Without a proof the first part of Lemma \ref{lema2}  is exposed as \text{A.11} on pages 732-733 of  \cite{Fl-Sed}; however, an inaccuracy is left in the case $N=j$.  For readers convenience,   we provide a sketch of a proof.

    Let $N\in\N_0$ be fixed. The coefficients under consideration have expressions in terms of $\Phi_k$ with $0\leq k\leq N$ only; therefore, we may assume that $\Phi(w)$ is a polynomial of degree $N$. Then $f(z)$ is well defined as an analytic function in a vicinity of the zero point. Thus, we may apply the Cauchy formula. Afterwards let  $\delta$ and $\delta_1$ be sufficiently small positive constants. Using a substitution $z=w/\Phi(w)$ and properties of the one-to-one conformal mapping of the vicinities of the zero points in the $z$- and $w$-complex planes,  we obtain
\begin{align*}
[z^N]\Big(\frac{z}{f(z)}\Big)^j&=
\frac{1}{2\pi i}\int_{|z|=\delta}\frac{dz}{f(z)^j z^{N+1-j}}\\
&=
\frac{1}{2\pi i}\int_{|w|=\delta_1}\frac{d\big(w/\Phi(w)\big)}{w^j\big(w/\Phi(w)\big)^{N+1-j}}\\
&=
\frac{1}{2\pi i}\int_{|w|=\delta_1}\frac{\Phi(w)^{N-j} dw}{w^{N+1}}-\frac{1}{2\pi i}\int_{|w|=\delta_1}\frac{\Phi(w)^{N-j-1} d\Phi(w)}{w^N}\\
&=
[w^N]\Phi(w)^{N-j}-\frac{1}{2\pi (N-j)i}\int_{|w|=\delta_1}\frac{d\Phi(w)^{N-j}}{w^{N}}\\
&=
[w^N]\Phi(w)^{N-j}-\frac{N}{2\pi (N-j)i}\int_{|w|=\delta_1}\frac{\Phi(w)^{N-j} dw}{w^{N+1}}\\
&=
\frac{j}{j-N}[w^N]\Phi(w)^{N-j}
\end{align*}
provided that $N\not=j$.

The same argument gives
\begin{align*}
[z^j]\Big(\frac{z}{f(z)}\Big)^j&=
\frac{1}{2\pi i}\int_{|w|=\delta_1}\frac{d w}{w^{j+1}}-\frac{1}{2\pi i}\int_{|w|=\delta_1}\frac{\Phi'}{\Phi}(w) \frac{dw}{w^j}\\
&=
-[w^{j-1}]\Big(\frac{\Phi'}{\Phi}(w)\Big).
\end{align*}

Finally, applying the previous substitution, we derive
\begin{align*}
[z^N]\log \frac{z}{f(z)}&=
\frac{1}{2\pi N i }\int_{|z|=\delta}\frac{1}{z^{N}}d\log\frac{z}{f(z)} \\
&=
-\frac{1}{2\pi N^2i}\int_{|w|=\delta_1}\frac{ d\Phi(w)^N}{w^N}\\
&=
-\frac{1}{N}[w^N] \Phi(w)^N.
\end{align*}

\end{proof}

We will apply the lemmas in a very particular case. Then the first power series coefficients of implicitly defined functions attain a simple form.
 Let ${\mathbf 1}\{\cdot\}$ stand for the indicator function.

\begin{lemma}\label{lema3}    Let $k,r, j\in\N$, $y=y(z)$ satisfy an equation
 \[
 y=z\bigg(\frac{1-y^r}{1-y}\bigg)^{1/r},
 \]
 and let $g(z):=z/y(z)$, then  the following assertions hold.

   $(I)$ \quad If $g(z)^j=:\sum_{N=0}^\infty g_N^{(j)} z^N$, then
      \[
  g_N^{(j)}= \frac{j}{j-N}\sum_{rl+m=N\atop l,m\in\N_0} {(N-j)/r\choose l}(-1)^l {m-1+(N-j)/r\choose m}
   \]
   for $N\in\N_0\setminus\{j\}$ and
  \begin{equation}
  g_j^{(j)}={\mathbf 1}\{j\equiv 0(\mathrm{mod}\, r)\}-\frac{1}{r}.
   \label{gjj}
\end{equation}

   $(II)$ \quad If $\log g(z)=:\sum_{N=1}^\infty b_N z^N$, then
  \[
  b_N=-\frac{1}{N}\sum_{rl+m=N\atop l,m\in\N_0} {N/r\choose l}(-1)^l {m-1+N/r\choose m}, \quad N\geq1.
    \]

 $(III)$ \quad If
 \[
      h(z):=\sum_{j=1}^r\frac{1}{j y(z)^j}=\sum_{N=-r}^\infty h_N z^N,
      \]
      then
       $h_{-r}=1/r$,
\[
  h_0=- \frac{1}{r}\sum_{j=2}^r\frac{1}{j}
\]
and
\[
h_N= \frac{N+r}{N} b_{N+r}
\]
for  $N=-r+1,-r+2,\dots$ and $N\not=0$.

 $(IV)$ \quad If
 \[
      \Lambda(z):=\bigg(z^r\sum_{j=1}^r\frac{j}{y(z)^j}\bigg)^{-1}=\sum_{N=0}^\infty \Lambda_N z^N
      \]
then $\Lambda_0=1/r$ and $\Lambda_N=-Nb_N/r$ for $N=1, 2,\dots$.
    \end{lemma}

\begin{proof} To prove $(I)$, combine Lemmas \ref{lema1} and  \ref{lema2} with  an equality
     \[
    [y^{N}]\bigg(\frac{1-y^r}{1-y}\bigg)^{\alpha}=\sum_{rl+m=N\atop l,m\in\N_0} {\alpha\choose l}(-1)^l {m-1+\alpha \choose m}, \quad N\in\N_0,\, \alpha\in\R.
   \]

For $\eqref{gjj}$, apply the second part of Lemma \ref{lema2} to obtain
\[
   g_j^{(j)}=[y^{j-1}]\bigg(\frac{y^{r-1}}{1-y^r}-\frac{1}{r(1-y)}\bigg)={\mathbf 1}\{j\equiv 0(\mathrm{\mod}\, r)\}-\frac{1}{r}.
\]

   Similarly, $(II$) follows from the last formula in Lemma \ref{lema2}.

Having in mind that $z^rh(z)$ has a power series expansion in $\C[[z]]$, we may apply the same principles. Using $\eqref{gjj}$, it is easy to check that
\begin{align*}
h_0&=\frac{1}{2\pi  i }\int_{|z|=\delta}\frac{ h(z) dz}{z} =
\sum_{j=1}^r\frac{1}{j} \frac{1}{2\pi  i }\int_{|z|=\delta}\frac{ g(z)^j d z}{z^{j+1}}\\
&=
\sum_{j=1}^r\frac{1}{j} g_j^{(j)}=- \frac{1}{r}\sum_{j=2}^r\frac{1}{j}.
\end{align*}
Further, we  observe that
\begin{equation}
   \sum_{j=1}^r \frac{1}{y^j}=\frac{1-y^r}{y^r(1-y)}=\frac{1}{z^{r}}.
\label{frac}
\end{equation}
Hence,
\[
      h'(z)=-\frac{y'}{y}(z) \sum_{j=1}^r \frac{1}{y^j}=-\frac{y'}{y}(z) \frac{1}{z^{r}}= \frac{g'}{g}(z) \frac{1}{z^{r}}-\frac{1}{z^{r+1}}, \quad z\not=0.
\]
This implies that
\begin{align*}
h_N&=\frac{1}{2\pi N i }\int_{|z|=\delta}\frac{ d h(z)}{z^{N}} \\
&=
\frac{1}{2\pi N i }\int_{|z|=\delta}\frac{ d (\log g(z))}{z^{N+r}}-\frac{1}{2\pi N i }\int_{|z|=\delta}\frac{ d z}{z^{N+r+1}} \\
 &= \frac{N+r}{N} b_{N+r}
\end{align*}
if $N\geq -r+ 1$ and $h_{-r}=1/r$.

To prove $(IV)$, we use relation (\ref{frac}) again. Differentiating it, we arrive at
\begin{align*}
          \Lambda(z)&=\frac{z}{r} \frac{y'}{y}(z)=\frac{1}{r}-\frac{z\big(\log g(z)\big)'}{r}\\
          &=\frac{1}{r}\bigg(1-\sum_{N=1}^\infty Nb_N z^N\bigg).
\end{align*}
The assertion $(IV)$ is evident now.

\end{proof}

\begin{corollary}\label{cor1} As above, let $g(z)=z/y(z)$. Then $g_0=1$, $g_1=-1/r$,
    \begin{equation*}
   g_N=  \frac{ \Gamma(N+(N-1)/r)}{(1-N)\Gamma(N+1)\Gamma((N-1)/r)}
   \end{equation*}
   if $2\leq N\leq r-1$, and
   \[
   g_r=  \frac{ \Gamma(r+1-1/r)}{(1-r)\Gamma(r+1)\Gamma(1-1/r)}+\frac{1}{r}.
   \]
   Moreover,
   $
    |g_N|\leq \frac{1}{N-1} r^{(N-1)/r}
$
if $N\geq 2$.
     \end{corollary}

\begin{proof} Apply $(I)$ of Lemma \ref{lema3} for $j=1$. If $2\leq N\leq r-1$, the relevant sum has the only nonzero summand corresponding to the pair $(l,m)=(0,N)$. A formula for $g_r$ has two summands giving the expression. If $N\geq 2$, then by Lemma \ref{lema3} and the Cauchy inequality,
  \[
    |g_N|=  \frac{1}{N-1}\Big|[y^N](1+y+\cdots+y^{r-1})^{(N-1)/r}\Big|\leq \frac{1}{N-1} r^{(N-1)/r}.
\]

\end{proof}

\begin{corollary} \label{cor2} We have
      \begin{equation}
  b_N=- \frac{\Gamma(N+N/r)}{N\Gamma(N+1)\Gamma(N/r)}
  \label{bN}
  \end{equation}
  if $1\leq N\leq r-1$ and $b_r=0$.

  Moreover,
  \[
  N|b_N|\leq 1\  \text{if}\ \:  N\leq r-1,
  \]
  \[
  b_N =O\left(N/r\right)\  \text{if}\ \:  r<N\leq 2 r-1,
  \]
and
   \begin{equation*}
          N|b_N|\leq  r^{N/r}\ \text{if}\ \; N\geq 1.
       \end{equation*}
     \end{corollary}

\begin{proof} Again, if $1\leq N\leq r-1$,  it suffices to observe that the relevant sum $(II)$ of  Lemma \ref{lema3} has the only nonzero summand corresponding to the pair $(l,m)=(0,N)$. A formula for $b_r$ has two subtracting summands.

If $N\leq r-1$, the given estimate follows from (\ref{bN}). If $r<N\leq 2r-1$, assertion $(II)$ in Lemma \ref{lema3} gives
     \begin{align*}
   b_N&=
   -\frac{1}{N}{N-1+N/r\choose N}+\frac{1}{r}{N-r-1+N/r\choose N-r}\\
   &=
   -\frac{1}{r} \prod_{k=2}^N\Big(1+\frac{N/r-1}{k}\Big)+\frac{N}{r^2} \prod_{k=2}^{N-r}\Big(1+\frac{N/r-1}{k}\Big)\\
   &\ll
       \frac{1}{r}\exp\bigg\{\Big(\frac{N}{r}-1\Big)\sum_{k=2}^N \frac{1}{k}\bigg\}\\
       &\ll
    \frac{1}{r}\exp\bigg\{\Big(\frac{N}{r}-1\Big)\log N\bigg\}\ll  \frac{ N}{r}.
   \end{align*}
   We have applied an inequality $\log(1+x)\leq x$ if $x>0$.

   Finally, by the Cauchy inequality,
   \begin{equation*}
          N|b_N|=
     \bigg| [y^{N}]\bigg(\frac{1-y^r}{1-y}\bigg)^{N/r}\bigg|=
      \big| [y^{N}](1+y+\cdots+ y^{r-1})^{N/r}\big|\leq
       r^{N/r}
       \end{equation*}
       if $N\geq 1$.
\end{proof}

\begin{lemma} \label{lema11}
	Let $x$ be the positive solution to the equation $\sum_{j=1}^r x^j=n$. If $r,n\in\N$ and $2\leq r\leq \log n$, then
	\begin{align*}
	x&=
	n^{1/r}-\frac{1}{r}  -\sum_{N=2}^{r}\frac{ \Gamma(N+(N-1)/r)}{(N-1)\Gamma(N+1)\Gamma((N-1)/r)}
	n^{-(N-1)/r}\\
	&\quad\    + \frac{1}{r}n^{-1+1/r}+O\left(\frac{1}{n}\right).
	\end{align*}
\end{lemma}

\begin{proof}
	The equation  defining  $x$ can be rewritten as
	\[
	x^{-1}=\bigg(\frac{1-x^{-r}}{1-x^{-1}}\bigg)^{1/r}n^{-1/r}.
	\]
	This gives the relation $y(n^{-1/r})=x^{-1}$, where $y=y(z)$ has been explored in Lemma \ref{lema3}. Consequently, we may apply the expansions of $g(z)$ given in $(I)$  with respect to powers of  $z=n^{-1/r}$.
	The first coefficients have been calculated in Lemma \ref{cor1}. It remains to estimate the remainder. Using also the obtained estimates, we have
	\[
	\sum_{N=r+1}^\infty |g_N| |z|^N\leq r^{-1-1/r}\sum_{N=r+1}^\infty  |r^{1/r} z|^N\leq \frac{|z|^{r+1}}{1-\sqrt[3]{3} {\re}^{-1}}
	\]
	if $|z|\leq {\re}^{-1}$.  Consequently,  we obtain
	\begin{align*}
	x&= n^{1/r} \sum_{N=0}^r g_N n^{-N/r}+ O\left(1/n\right)\\
	&=
	n^{1/r}-\frac{1}{r}  -\sum_{N=2}^{r}\frac{ \Gamma(N+(N-1)/r)}{(N-1)\Gamma(N+1)\Gamma((N-1)/r)}
	n^{-(N-1)/r} \\
	& \quad\
	+ \frac{1}{r}n^{-1+1/r}+O\left(1/n\right)
	\end{align*}
	as desired.
\end{proof}

\textit{Proof of Theorem~\ref{Theorem 1}.} Let us preserve the  notation  introduced in Lemma \ref{lema3}.
First of all, we seek  a simple expression containing the first terms in an expansion of
\[
  K(z):=   \sum_{j=1}^r \frac{1}{j y(z)^j}-n\log \frac{z}{y(z)}= h(z)-n\log g(z).
\]
Let $D(x):=\exp\left\{\sum_{j=1}^r x^j/j\right\}$, we have
\begin{equation}
\log D(x)-n\log x= K(n^{-1/r})-\frac{n\log n}{r}.
\label{Dx}
\end{equation}
Define the  functions $R(z)$ and $K_r(z)$ by
\begin{equation}\label{5.41}
K(z)=\sum_{N=-r+1}^0 h_N z^N-n\sum_{N=1}^{r-1} b_N z^N +R(z)=K_r(z)+R(z)
\end{equation}
We claim  that $R(z)=O(|z|+n|z|^{r+1})$ if $|z|\leq \re^{-1}$ implying
\begin{equation}
  R(n^{-1/r})=O\left(n^{-1/r}\right).
  \label{Rz}
  \end{equation}
 for $r\leq \log n$. Indeed, by $(III)$ of Lemma \ref{lema3} and the estimates in Corollary \ref{cor2}, we have
       \begin{align*}
          \sum_{N=1}^\infty |h_N||z|^N&=
           \bigg(\sum_{N=1}^{r-1}+\sum_{N=r}^\infty\bigg)\frac{N+r}{N}|b_{N+r}||z|^N\\
          &\ll
            \sum_{N=1}^{r-1}\frac{N+r}{r}|z|^N +\sum_{N=r}^\infty\frac{r}{N}\big(r^{1/r}|z|\big)^N
          \ll|z|
          \end{align*}
       if $|z|\leq \re^{-1}$. Similarly,
       \[
          \sum_{N=r+1}^\infty |b_N||z|^N\ll |z|^{r+1}
          \]
       if $|z|\leq \re^{-1}$. The last two estimates yield our claim and (\ref{Rz}).

     For the main term, we obtain from  Lemma \ref{lema3} that
     \begin{align*}
     K_r(n^{-1/r})&= h_0+\sum_{N=-r+1}^{-1} h_N n^{-N/r}- \sum_{N=1}^{r-1} b_N n^{(r-N)/r}+h_{-r} n\\
     &=
     h_0 -\sum_{N=1}^{r-1} \frac{r-N}{N} b_{r-N} n^{N/r}- \sum_{N=1}^{r-1} b_N n^{(r-N)/r}+h_{-r} n\\
     &=
     h_0 -r\sum_{N=1}^{r-1} \frac{1}{N} b_{r-N} n^{N/r}+h_{-r} n\\
     &=
- \frac{1}{r}\sum_{j=2}^r\frac{1}{j}+ r \sum_{N=1}^{r-1}\frac{1}{N(r-N)}\, \frac{\Gamma(N+N/r)}{\Gamma(N+1)\Gamma(N/r)}n^{(r-N)/r}+\frac{n}{r}.
     \end{align*}

     It remains to approximate
     \[
     \Big(\frac{1}{\lambda(x)}\Big)^{1/2}= \frac{1}{\sqrt n} \Lambda(n^{-1/r})^{1/2}=\frac{1}{\sqrt {n r}}\bigg(1-\sum_{N=1}^\infty Nb_N n^{-N/r}\bigg)^{1/2}.
     \]

    By virtue of Corollary \ref{cor2}, $N|b_N|\leq 1$ if $N\leq r$ and  $N|b_N|\leq r^{N/r}$ if $N\geq 1$.
       Thus, if $2\leq r\leq \log n$, then
     \[
        \sum_{N=1}^\infty N|b_N|n^{-N/r}\leq (5/2) n^{-1/r}\leq (5/2){\re}^{-1}<1.
        \]
 Consequently,
\[
     \Big(\frac{1}{\lambda(x)}\Big)^{1/2}= \frac{1}{\sqrt {n r}}\big(1+O\left(n^{-1/r}\right)\big).
\]

     We  now  return to probabilities.
     Applying \eqref{P-ell}, \eqref{fullx},   \eqref{Dx}, \eqref{5.41},  \eqref{Rz}, and the last estimate, we obtain
\begin{align*}
   \nu(n,r)&=\frac{1}{{\sqrt{ 2\pi nr}}} n^{-n/r}\exp\big\{K_r(n^{-1/r})\big\}\big(1+O\left(n^{-1/r}\right)\big)
\end{align*}

for  all $2\leq r\leq \log n$.

\qed

\section{Proof of Theorem~\ref{thm2}}\label{s:4}

\bigskip

\textbf{Theorem~\ref{thm2}.} \textit{Let $u=n/r$. If $\sqrt{n\log n}\leq r\leq n$, then}
\[
\nu(n,r)=\varrho(u)\bigg(1+O\left(\frac{u\log(u+1)}{r}\right)\bigg).\\
\]

\begin{proof} The idea is to use the Cauchy integral
\begin{equation}
   \nu(n,r)={1\over 2\pi i}\int_{|z|=\alpha}
  \exp\left\{\sum_{j=1}^r\frac{z^j}{j}\right\} {dz\over z^{n+1}}
\label{C-integr1}
   \end{equation}
with  $\alpha:=\re^{\xi(u)/r}$ which is a good approximation of the saddle-point.
Here, $\xi:=\xi(u)$ is provided in Lemma~\ref{xi}. Such a choice  relates $Q(z)$ defined in (\ref{Qz-naujas}) with  the Laplace transform of the Dickman function. Namely, if $z=\re^{-s/r}$, $s=-\xi+ir t=:-\xi+i\tau $, and $|t|\leq\pi$, then (see Lemmas~\ref{IT} and~\ref{hatrho})
\begin{equation}
    Q\big(\re^{-s/r}\big)=
    \exp\big\{us+I(-s)+T(-s)\big\}=
    \hat\varrho(s)\exp\big\{-\gamma+us+T(-s)\big\},
    \label{Qz}
\end{equation}
where $T(-s)$ is examined in Lemma \ref{Tz}.

Let us introduce the following vertical line segments in the complex plane:
\[
\Delta_0:=\{s=-\xi+i\tau:\; |\tau|\leq \pi\},\qquad  \Delta_1:=\{s=-\xi+i\tau:\; \pi\leq \tau\leq r\pi\},
\]
\[
 \Delta_2:=\{s=-\xi+i\tau:\; -\pi r\leq \tau\leq -\pi\},\qquad \Delta=\{s=-\xi+i\tau:\; |\tau|\leq r\pi\},
 \]
 and $\Delta_\infty=\{s=-\xi+i\tau:\; |\tau|\geq r\pi\}$. Taking into account (\ref{Qz}), we have from (\ref{C-integr1})
 \begin{align*}
\Pr\Big(\sum_{j=1}^rjZ_j=n\Big)&=\frac{1}{2\pi i}\int_{|z|=y} \frac{Q(z) dz}{z}\nonumber\\
&=
\frac{\re^{-\gamma}}{r}\frac{1}{2\pi i}\int_{\Delta}\re^{us}\hat\varrho(s) ds+
\frac{\re^{-\gamma}}{2\pi r i}\int_{\Delta}\re^{us}\hat\varrho(s)\big(\re^{T(-s)}-1\big) ds\\
&=:
I+J.
\end{align*}

Using Lemmas \ref{xi}, \ref{rho}, and  \ref{rholap} for the case $|\tau|\geq \pi r>1+u\xi$, we obtain
\begin{align*}
I&=\frac{\re^{-\gamma}\varrho(u)}{r}- \frac{1}{2\pi i r u} \int_{\Delta_\infty}\hat\varrho(s) d(\re^{us})\\
&=
\frac{\re^{-\gamma}\varrho(u)}{r}+ O\left(\frac{\re^{-u\xi}}{ur^2}\right) + \frac{1}{2\pi i u r}\int_{\Delta_\infty}\re^{us}\hat\varrho(s)\frac{\re^{-s}-1}{s} ds\\
&=
\frac{\re^{-\gamma}\varrho(u)}{r}+ O\left(\frac{\re^{\xi-u\xi}}{ur^2}\right) \\
&=
\frac{\re^{-\gamma}\varrho(u)}{r}+ O\left(\frac{\varrho(u) \re^{\xi-I(\xi)}}{r^2}\right) \\
&=
\frac{\re^{-\gamma}\varrho(u)}{r}\Big(1+O\left(1/r\right)\Big).
\end{align*}
In the last step, we have used the fact that $I(\xi)\sim \re^\xi/\xi$ as $\xi\to\infty$.

 The next task is to estimate $J$. If $s\in\Delta$ then, by Lemma \ref{Tz}, $T(-s)=O\left(1\right)$ and  $\exp\{T(-s)\}= 1+ O\left(T(-s)\right)$. Let us split $J$
 into the sum of three integrals $J_k$  over the strips $\Delta_k$, where  $k=0,1$ and 2, respectively. If  $s\in \Delta_0$ then  $T(-s)=O(1+u\log u)/r$.
  Therefore, using Lemmas \ref{xi}, \ref{rho}, and  \ref{rholap}, now for the case $|\tau|\leq \pi$, we derive
\begin{align*}
   J_0&\ll \frac{(1+u\log u)}{r^2}\int_{\Delta_0} \big|\hat\varrho(s)\re^{us}\big||ds|\\
   &\ll
   \frac{(1+u\log u)\varrho(u)\sqrt u}{r^2}
   \int_{-\pi}^{\pi}\re^{-\tau^2u/(2\pi^2)} d\tau\\
&\ll
   \frac{(1+u\log u)\varrho(u)}{r^2}.
\end{align*}
Further,
\begin{align*}
   J_1&=\frac{1}{2\pi i u r}\int_{\Delta_1} \hat\varrho(s)\big(\re^{T(-s)}-1\big) d \re^{us}\\
   &\ll
   \frac{\re^{-u\xi}}{u r}\big|\hat\varrho(-\xi+\pi i)T(\xi-\pi i)\big| +\frac{\re^{-u\xi}}{u r}\big|\hat\varrho(-\xi+\pi r i)T(\xi-\pi r i)\big|\\
   &\quad\ +
   \frac{1}{u r}\int_{\Delta_1} \re^{us}\Big(\hat\varrho(s)'\big(\re^{T(-s)}-1\big)- \hat\varrho(s)T'(-s)\re^{T(-s)}\Big)d s\\
   &=: L_1+L_2+O\left(\frac{L_3}{u r}\right).
   \end{align*}

To estimate $L_1$, we combine the first estimate of $\hat\varrho(s)$ given in Lemma \ref{rholap} with  Lemmas~\ref{xi} and \ref{rho}. So we obtain
  \[
      L_1\ll \frac{(1+u\log u)}{u r^2}\re^{-u\xi+I(\xi)}\ll \frac{\varrho(u)(1+u\log u)}{r^2}.
\]
Similarly, the second estimate in Lemma \ref{rholap} leads to
  \[
      L_2\ll \frac{\re^{-u\xi}}{u r^2}\ll \frac{\varrho(u)\re^{-I(\xi)}}{r^2\sqrt u}\ll \frac{\varrho(u)}{r^2}.
      \]

  Estimation of the integral $L_3$ is more subtle. It uses an estimate
  \[
 1- b(-s/r)-T(-s)\ll \frac{\re^\xi}{r}+\left|\frac{s}{r}\right|^2
  \]
  following from \eqref{Tz1}, \eqref{Bern} and the asymptotic formula  $ b(v)=1+v/2+O\left(v^2\right) $  for  $|v|\leq \pi\sqrt2$.
 We have
  \begin{align*}
L_3&=
\int_{\Delta_1}\re^{us} \frac{\re^{-s}-1}{s}\hat{\varrho}(s)\left(1+\frac{s}{r(1-\re^{s/r})}\re^{T(-s)}\right) ds\\
&=
\int_{\Delta_1}\re^{us} \frac{\re^{-s}-1}{s}\hat{\varrho}(s)\left(1- b(-s/r)\re^{T(-s)}\right)ds\\
&=
\int_{\Delta_1}\re^{us} \frac{\re^{-s}-1}{s}\hat{\varrho}(s)\bigg(\big(1- b(-s/r)-T(-s)\big)+O\Big(\frac{s T(-s)}{r}+T(-s)^2\Big)\bigg)
ds\\
&\ll \re^{-u\xi}\int_{\Delta_1} \frac{|\re^{-s}-1|}{|s|}|\hat{\varrho}(s)|\Big(\frac{e^\xi}{r}+\frac{|s|^2}{r^2}\Big)
|ds|.
\end{align*}
Using the  two different estimates of $\hat\varrho(s)$ on the line segments $\Delta_{11}:=\{s\in\Delta_1:\; |\Im s|\leq 1+u\xi\}$ and
 $\Delta_{12}:=\Delta_1\setminus \Delta_{11}$ given by Lemma \ref{rholap}, we proceed as follows:
 \begin{align*}
L_3&\ll \exp\bigg\{ -u\xi+I(\xi)-\frac{u}{\pi^2+\xi^2} +\xi\bigg\}\int_{\Delta_{11}}\frac{1}{|s|} \bigg(\frac{\re^{\xi}}{r}+ \frac{|s|^2}{r^2}\bigg)|d s|\\
&\qquad
+\exp\big\{-u\xi +\xi\big\} \int_{\Delta_{12}}\frac{1}{|s|^2} \bigg(\frac{\re^{\xi}}{r}+ \frac{|s|^2}{r^2}\bigg)|d s|\\
&\ll \exp\bigg\{ -u\xi+I(\xi)-\frac{u}{\pi^2+\xi^2} +2\xi\bigg\}\frac{1+\xi}{r}+\frac{\exp\big\{-u\xi +\xi\big\}}{r}\\
&\ll
\frac{\varrho(u)\sqrt u\log(u+2)}{r}.
\end{align*}

Collecting the obtained estimates, we obtain
\[
   J_1= L_1+L_2+O\left(\frac{L_3}{u r}\right)\ll \frac{\varrho(u)(1+u\log u)}{r^2}.
   \]
   The same holds for integral $J_2$. Consequently,
 \[
 \Pr\Big(\sum_{j=1}^rjZ_j=n\Big)=I+J_0+J_1+J_2=
\frac{\re^{-\gamma}\varrho(u)}{r}\Big(1+O\left(\frac{1+u\log u}{r}\right)\Big).
\]
\end{proof}

\medskip

\section{Proofs of Theorem~\ref{1cor} and Corollary~\ref{2cor}}

\bigskip

\textbf{Theorem~\ref{1cor}.} \textit{If $ 1\leq r\leq c n(\log n)^{-1}(\log\log n)^{-2}$,
where $c=1/(12\pi^2 \re)$ and $n\geq 4$, then
\begin{equation*}
\nu(n,r)=
\varrho(u) \exp\Big\{\frac{u\xi(u)}{2r}\Big\}\left(1+O\left(\frac{u\log^2(u+1)}{r^2}+\frac{1}{u}\right)\right),
\end{equation*}}
where $u=n/r.$

\noindent\textbf{Corollary~\ref{2cor}.} \textit{The formulas presented in Theorems~\ref{thm1} and~\ref{1cor} remain to hold if the upper bound of $r$ is substituted by $n$.}

We first prove a key lemma.

\begin{lemma}\label{4lema}
Denote $x$ to be the positive solution to the equation  $\sum_{j=1}^rx^j=vr$, where $v\geq 1$ is a continuous parameter and $r\in\N$. Let $x(v):=x$ and $\lambda(x):=\sum_{j=1}^rjx^j$. If $ v\in(1, \re^r]$, $\xi:=\xi(v)$, and $\xi':=\xi'(v)$ (see Lemma~\ref{xi}), then
 \begin{equation}
x=x(v)= \exp\Big\{\frac{\xi}{r}\Big\}+O\left(\frac{\log (v+1)}{r^2}\right)
\label{xxii}
\end{equation}
and
\begin{equation}
\frac{r}{\lambda(x)}=\frac{x'}{x}(v)= \frac{\xi'}{r}\left(1+O\left(\frac{ \log (v+1)}{r}\right)\right).
\label{xxiii}
\end{equation}
\end{lemma}
   \begin{proof}  One may skip the  trivial case when  $r$ is bounded.
   From the definition of $x(v)$ and (\ref{lambda}), we have
   \begin{equation}
   0<x'(q)=\frac{rx(q)}{\lambda(x(q))}\leq \frac{x(q)}{rq\big(1-\log^{-1} 3\big)}\ll\frac{x(q)}{rq}
   \label{xderiv}
   \end{equation}
if $q\geq 3$. The same holds if $1\leq q\leq 3$. Indeed, in this case it suffices to apply the trivial estimate
$
   \lambda(x(q))\geq  r^2/2\geq r^2 q/6.
$

  As a function of $q\geq 1$,  $\exp\big\{\xi(q)/r\big\}$ is also strictly increasing; therefore, given any  $v\geq 1$ and the
  value $\xi=\xi(v)$, we can find $w\geq1$ such that
   \[
   x(w)=\exp\big\{\xi/r\big\}.
   \]
Now
\begin{equation}
x-\exp\big\{\xi/r\big\}=x(v)-x(w)\ll (v-w) x'(q),
\label{xu=}
\end{equation}
 where $q$ is a point between the $v$ and $w$, irrespective of their relative position on the real line.

  Using  (\ref{xr}) with  $w$ instead of $v$, we have
  \[
    x^r(w)-1=  \re^\xi-1= rw\big(1-x(w)^{-1}\big)= rw\big(1-\re^{-\xi/r}\big).
    \]
By the definition of $\xi$ and Lemma 1, we obtain from the last relation that $v\xi= rw\big(\xi/r+ B(\xi/r)^2\big)$ with $|B|\leq 1/2$. Hence,
 \begin{equation}
     |v-w|\leq w\xi/(2r).
 \label{uminv}
 \end{equation}

 If $v\leq 3$ and $r\geq 1$, then $0.09w<w(1-\xi(3))\leq 2v\leq 6$ and $v-w\ll r^{-1}$. Therefore,  estimates (\ref{xderiv}) and (\ref{xu=}) imply
 \[
 x-\exp\big\{\xi/r\big\}\ll r^{-2},
 \]
 as desired in (\ref{xxii}).

 If $v\geq 3$, then by virtue of $\xi\sim\log v$ as $v\to\infty$ and $\log v\leq r$, we obtain from (\ref{uminv}) that $|v-w|\leq (3/4)w$ if $r$ is
  sufficiently large. Hence, $(4/7) v\leq w\leq 4v$ and $(4/7)v\leq q\leq 4 v$. By Lemma \ref{x}, this gives
   $x(q)\leq x(4v)=O\left(1\right)$. Formula (\ref{xxii}) again follows from (\ref{xderiv}) and (\ref{xu=}).

 To derive approximation (\ref{xxiii}) of the logarithmic derivative, we use similar arguments. Firstly,
  given   $v\geq 3$,  we define  $y>1$ such that
   $x=\re^{\xi(y)/r}$ and claim that
\begin{equation}
     \xi=\xi(y)\big(1+O\left(1/r\right)\big).
     \label{y}
     \end{equation}
     Indeed, if also $v\leq \re^r$, then an observation in the proof of Lemma \ref{x} gives us $\xi(y)=r\log x\leq \log (vr)
     \leq (6/5) r$ if $r$ is sufficiently large.
By the definitions and inequalities
\[
  0<\frac{t}{1-\re^{-t}}-1=\frac{t-1+\re^{-t}}{1-\re^{-t}}\leq \frac{t^2/2}{t-t^2/2}\leq \frac{3t}{2}
\]
if $0<t\leq 6/5$, we further obtain
     \begin{equation}
  v=\frac{x}{r} \frac{x^r-1}{x-1}=
     \frac{\re^{\xi(y)}-1} {\xi(y)} \frac {\xi(y)/r}{1-\re^{-\xi(y)/r}}=     y\Big(1+\frac{B\xi(y)}{r}\Big)
     \label{u}
     \end{equation}
     with $0<B\leq 3/2$.
      Hence, $ 15/14\leq (5/14)v< y\leq v $
       and also $ \xi'(q)\ll 1/q\ll 1/y$ for all $q\in[y, v]$, by Lemma \ref{xi}. Inserting this and (\ref{u})
      into   $ \xi-\xi(y)=(v-y)\xi'(q) $ with some $q\in[y, v]$,  we complete the proof of (\ref{y}).

      Let us keep in mind the bound $y\geq 15/14$ and  return to the logarithmic derivative. It follows from (\ref{Lambda}) and (\ref{xr}) that
\[
\frac{x'}{x}\bigg(\frac{x^r}{x^r-1}-\frac{1}{r(x-1)}\bigg)= \frac{1}{rv}.
 \]
Now, the idea is to rewrite  the quantity in large parentheses via $\xi(y)$, then use inequality (\ref{y}) to approximate it by $\xi$ and $\xi'$.

The inequality $0<t^{-1}-(\re^t-1)^{-1}<1$ applied with $t=\xi(y)/r$ gives $(r(x-1))^{-1}=1/\xi(y)+O\left(1/r\right)$; therefore,
\begin{equation*}
\frac{x'}{x}  \left(\frac{1+y\xi(y)-y}{y\xi(y)}+O\left(\frac{1}{r}\right)\right)= \frac{1}{rv}.
 \end{equation*}
 Because of (\ref{deriv}), the first ratio inside the parentheses is $1/(y\xi'(y))$ which, by Lemma \ref{xi},
 satisfies an inequality
 \[
      \frac{1}{y\xi'(y)}\geq \frac{y\log y-y+1}{y\log y}=:l(y)\geq l\Big(\frac{15}{14}\Big)>0.
 \]
Now using (\ref{u}) and (\ref{y}), we obtain
\begin{align*}
\frac{x'}{x} &= \frac{1}{rv} \frac{y\xi(y)}{1+y\xi(y)-y}\Big(1+O\left(\frac{1}{r}\right)\Big)\\
&= \frac{1}{rv} \frac{v\xi}{1+v\xi-v}\Big(1+O\left(\frac{\log v}{r}\right)\Big)\\
&= \frac{\xi'}{r}\Big(1+O\left(\frac{\log v}{r}\right)\Big)
 \end{align*}
 if $3\leq v\leq \re^r$.

 In the case $1<v\leq 3$, we have from (\ref{xxii})
 \begin{align*}
 \lambda(x)&=\sum_{j=1}^r j\bigg(\re^{\xi/r}+O\left(\frac{1}{r^2}\right)\bigg)^j=
 r\sum_{j=1}^r \frac{j}{r}\re^{\xi j/r}+O\left(r\right)=r^2\int_0^1t\re^{t\xi} dt+O\left(r\right)\\
 &=
 \frac{r^2}{\xi}(v\xi+1-v)+O\left(r\right)= \frac{r^2}{\xi'}+O\left(r\right).
 \end{align*}
 Hence,
 \[
    \frac{x'}{x}=\frac{r}{\lambda(x)}=\frac{\xi'}{r}\Big(1+O\left(\frac{1}{r}\right)\Big).
    \]

\end{proof}

Now, having the lemma, we can present the following proof.

\begin{proof} [Proof of Theorem~\ref{1cor}] Recall Lemma~\ref{IT} and let 
$$
Q(x)=\frac{1}{x^n}\exp\left\{ \sum_{j=1}^r \frac{x^j-1}{j} \right\},
$$
and $\lambda(x)=\sum_{j=1}^rjx^j$. Thus,
\begin{align}
\log Q(x)&=
-n\log x+\int_1^x\sum_{j=1}^r t^{j-1} dt
=-n\log x+\int_1^x\frac{t^r-1}{t-1} dt\nonumber\\
&=
-n\log x+\int_0^{r\log x}\frac{\re^v-1}{v} \frac{v}{r} \frac{dv}{1-\re^{-v/r}}\nonumber\\
&=
-u r\log x +I(r\log x)+T(r\log x).
\label{Tx}
\end{align}
Observe that relation (\ref{xr}), rewritten as
\[
   \re^{r\log x}=1+\Big(\frac{u(1-x^{-1})}{\log x}\Big)(r\log x)=:1+u'(r\log x),
   \]
gives $\xi(u')=r\log x$ with the uniquely defined $u'=u(1-x^{-1})/\log x\leq u$. Hence, by virtue of monotonicity, $r\log x\leq\xi(u)=\xi$ if $x\geq 1$. Therefore,
\begin{align*}
-u r\log x +I(r\log x)
&=
-u \xi +I(\xi)+u(\xi- r\log x) -\int_{r\log x}^\xi\frac{\re^{t}-1}{t} dt\\
&=
-u \xi +I(\xi)+(\xi- r\log x)\bigg(u-\frac{\re^{t_0}-1}{t_0}\bigg)
\end{align*}
with a $t_0\in[r\log x,\xi]$  and, consequently, if $1\leq u\leq \re^r$,
\[
(\re^{t_0}-1)/t_0\in \big[(x^r-1)/(r\log x), u\big]=[u+O\left(u\xi/r\right), u].
\]
In the last step, we have applied (\ref{xxii}) in the form
\[
    x^r=(1+u\xi)\big(1+O\left(r^{-1}\log(u+1)\right)\big).
\]
This yields
\begin{equation}
-u r\log x +I(r\log x)=-u \xi +I(\xi)+O\left(\frac{u\log^2(u+1)}{r^2}\right).
\label{urlog}
\end{equation}
If $1\leq u\leq \re^r$,  by Lemma~\ref{xi}, we have  $r\log x\leq \xi \leq  2\log u$. Thus, we may apply estimate (\ref{Tz2}) in Lemma~\ref{Tz} to obtain
\begin{align*}
   T(r\log x)&=\frac{x^r}{2r}+O\left(\frac{\log (u+1)}{r}\right)+O\left(\frac{u\log^2(u+1)}{r^2}\right)\\
   &=
   \frac{u\xi}{2r}+O\left(\frac{\log (u+1)}{r}\right)+O\left(\frac{u\log^2(u+1)}{r^2}\right).
\end{align*}
Inserting this and (\ref{urlog}) into expression (\ref{Tx}), we deduce a relation
\begin{equation*}
\log Q(x)=-u \xi +I(\xi)+\frac{u\xi}{2r}+O\left(\frac{\log (u+1)}{r}\right)+O\left(\frac{u\log^2(u+1)}{r^2}\right),
\end{equation*}
which is non-trivial if  $\log u=o(r)$. Combining this with (\ref{xxiii}) and Lemma \ref{rho}, we arrive at
\begin{align}
\frac{Q(x)}{\sqrt{2\pi\lambda(x)}}&=
\frac{\re^{-\gamma}}{r}\frac{\sqrt{\xi'}}{\sqrt{2\pi}}\exp\bigg\{\gamma-u\xi +I(\xi)+\frac{u\xi}{2r}\bigg\}\nonumber\\
&\quad\ \cdot\left(1+O\left(\frac{\log (u+1)}{r}\right)+O\left(\frac{u\log^2(u+1)}{r^2}\right)\right)\nonumber\\
&=
\frac{\re^{-\gamma}}{r}\varrho(u)\re^{u\xi/(2r)}\left(1+O\left(\frac{1}{u}\right)
+O\left(\frac{u\log^2(u+1)}{r^2}\right)\right)
\label{Qlambda}
\end{align}
if $n^{1/3}\log^{2/3} (n+1)\leq r\leq n$. If, in addition, $r\leq cn(\log n)^{-1}(\log\log n)^{-2}$, then  by Theorem \ref{thm1}, the ratio on the left hand side approximates the probability $\Pr\Big(\sum_{j=1}^rjZ_j=n\Big)$. Recalling (\ref{P-ell}), we complete the proof of the corollary.
\end{proof}

\begin{proof}[Proof of Corollary~\ref{2cor}] By relation (\ref{Qlambda}), the result of Theorem \ref{thm2} can be exposed as (\ref{fullx}) provided that $\sqrt{n\log n}\leq r\leq n$. Then the assertions of Theorems \ref{thm1} and \ref{thm2} can be joined up. Now, formula (\ref{fullx}), valid for $1\leq r\leq n$, and (\ref{Qlambda}) justify (\ref{nu-cor1}) for $n^{1/3}\log^{2/3}(n+1)\leq r \leq n$, one only needs to repeat the argument for Theorem~\ref{1cor}.
\end{proof}

\medskip

\section{Proofs of Theorems~\ref{thm3.2},~\ref{theorem3.3},~\ref{theorem3.4}, and Corollary~\ref{corollary3.5}}

\bigskip

\textbf{Theorem~\ref{thm3.2}.}
	\textit{Let $(t)_+=\max\{t,0\}$, $t\in\R.$
	For all $1\leq r<n/2$ and $\alpha>1$, we have 
	\begin{align*}
	|\kappa(n,r)-\re^{-H_r}| &\leq  \frac{\pi\re^4\alpha^{2r-n+3/2}}{n^2(\alpha-1)^2}\exp\left\{\sum_{j=1}^r\frac{\alpha^j-2}{j}+E(r,\alpha)\right\}\nonumber\\
	&\quad\ +\frac{4\re\alpha^{2r-n+2}}{\pi n^2r(\alpha-1)^3}\exp\left\{-\frac{\alpha(\alpha^r-1)}{2r(\alpha-1)}-H_r\right\}, 	
	\end{align*}
	where $$
	E(r, \alpha)=-\frac{2}{r}\frac{\alpha^{r+1}}{\alpha-1}\left(\frac{\pi^{-2}}{1+(r\alpha-r)^2}-\alpha^{-r/2} \right)_+ +\min\{2r\log\alpha, 2\log(\re r)\}.
	$$}

\noindent \textbf{Theorem~\ref{theorem3.3}.} \textit{Let $\varepsilon$ and $\delta$ be an arbitrary but fixed positive numbers, and $u=n/r$. We have
\begin{equation}
\kappa(n,r)=\re^{-H_{r}}+O_{\varepsilon, \delta} \left(\frac{\varrho(u)}{r}\exp\left\{-\frac{2u\left(1-\delta\right)}{\pi^2(\log (1+u))^2}\right\}\right)\label{thm121}
\end{equation} 
if $(\log n)^{3+\varepsilon}\leq r< n$. Moreover, 
\begin{equation}
\kappa(n,r)=\re^{-H_{r}}+O\left(\frac{\varrho(u)u^{u/r}}{r}\right) \label{thm122}
\end{equation}
if $\log n\leq r< (\log n)^{3+\varepsilon}$. Here $\varrho(u)$ is the Dickman function.}

\bigskip

\noindent\textbf{Theorem~\ref{theorem3.4}.} \textit{Let $\varepsilon$ be an arbitrary but fixed positive number, and $u=n/r$, then
\begin{equation}\label{131}
\kappa(n,r)=\re^{-H_r}+O_\varepsilon\Bigg(\frac{\nu(n,r)}{r}\exp\left\{-\frac{u^{1-4/r}(1-\varepsilon)}{4\pi^2(\log (u+1))^2}\right\}\Bigg)
\end{equation}
if $5\leq r< n$, and
\begin{equation}\label{132}
\kappa(n,r)=\re^{-H_r}+O\Bigg(\nu(n,r)n^{5/2}\Bigg)
\end{equation}
if $2\leq r<5$.}

\bigskip

\noindent\textbf{Corollary~\ref{corollary3.5}.} For $2\leq r\leq \log n$, we have
	\[
	\kappa(n,r)=\re^{-H_r}+O\left(\exp\left\{-\frac{n}{r}\log\frac{n}{\re}+\frac{n}{\log n}+\frac{3n}{(\log n)^2}\right\}\right).\\
	\]

The next lemma is an essential part of the proofs. We use the following notation
$$
(a)_+=\max\{a,0\}
$$
if $a\in\R$.
 
\begin{lemma} \label{trig}
If $1/r\leq\left|t\right|\leq \pi$, $r\in\N$, and $\alpha>1$, then

\begin{align}
\sum_{j=1}^r\frac{1-\alpha^j\cos(t j)}{j}&\leq  \sum_{j=1}^r\frac{\alpha^j-1}{j}-\frac{2}{r}\frac{\alpha^{r+1}}{\alpha-1}\left(\frac{\pi^{-2}}{1+(r\alpha-r)^2}-\alpha^{-r/2}\right)_+\nonumber\\
&\quad\ +\log\frac{\left|\alpha -\re^{it}\right|}{\alpha-1}+\min\{2r\log\alpha,\ 2\log(\re r)\}+4 \label{trig1}
\end{align}
and
\begin{equation}\label{trig2}
\sum_{j=1}^r\frac{1-\alpha^j\cos(t j)}{j}\leq -\frac{1}{2r}\frac{\alpha(\alpha^r-1)}{\alpha-1}+1
\end{equation}
if $|t|\leq 1/r$.
\end{lemma}

\begin{proof} Note that 

\begin{equation} \label{cos} 
2t^2/\pi^2\leq 1-\cos t\leq t^2/2 
\end{equation}
if $|t|\leq \pi$. Estimate \eqref{trig2} we get is not difficult. If $|t|\leq 1/r$, we have

\begin{align*}
\sum_{j=1}^r\frac{1-\alpha^j\cos(t j)}{j}&\leq -\sum_{j=1}^r\frac{\alpha^j-1}{j}+\sum_{j=1}^r\frac{\alpha^j\left(1-\cos(j/r)\right)}{j}\\
&\leq -\sum_{j=1}^r\frac{\alpha^j-1}{j}+\frac{1}{2r^2}\sum_{j=1}^rj\alpha^j\\
&\leq -\frac{1}{2r}\sum_{j=1}^r \alpha^j+1.
\end{align*}

To obtain estimate \eqref{trig1}, we use the strategy found in \cite[p.~407]{GT}. Notice that
\begin{align}
\sum_{j=1}^r\frac{1-\alpha^j\cos(t j)}{j}&=\sum_{j=1}^r\frac{\alpha^j-1}{j}-\sum_{j=1}^r\frac{\alpha^j\left(1+\cos(t j)\right)-2}{j}.\label{visa}
\end{align}
Thus, we evaluate only the second sum
\begin{align}
& \sum_{j=1}^r\frac{\alpha^j\left(1+\cos(t j)\right)-2}{j}\nonumber\\ 
&\geq  \sum_{j=1}^{[r/2]}\frac{\cos(t j)-1}{j}+\frac{1}{r}\sum_{j=[r/2]+1}^r\alpha^j\left(1+\cos(t j)\right)-\sum_{j=[r/2]+1}^r\frac{2}{j}\nonumber\\
&\geq \sum_{j=1}^{r}\frac{\cos(t j)-1}{j}+\frac{1}{r}\Re\sum_{j=[r/2]+1}^r \alpha^j\left(1+\re^{it j}\right)-4.\label{Re}
\end{align}
To obtain $-4$, we used inequalities $\log(1+r) \leq \sum_{j=1}^r1/j\leq 1+\log r$. Let us turn our attention to the most difficult part, namely, to the sum
\begin{align}
&\Re\sum_{j=[r/2]+1}^r \alpha^j\left(1+\re^{it j}\right)\nonumber\\
&=\frac{\alpha^{r+1}-\alpha^{[r/2]+1}}{\alpha-1}+\Re\frac{(\alpha\re^{it})^{r+1}-(\alpha\re^{it})^{[r/2]+1}}{\alpha\re^{it}-1}\nonumber\\
&= \frac{\alpha^{r+1}}{\alpha-1}\left(1-\alpha^{[r/2]-r}+\Re \frac{\re^{it(r+1)}-\alpha^{[r/2]-r}\re^{it([r/2]+1)}}{\alpha\re^{it}-1}\left(\alpha-1\right)\right)\nonumber\\
&\geq  \frac{\alpha^{r+1}}{\alpha-1}\left(1-2\alpha^{[r/2]-r}+\Re \frac{\re^{it(r+1)}\left(\alpha-1\right)}{\alpha\re^{it}-1}\right)_+\nonumber\\
&\geq  \frac{\alpha^{r+1}}{\alpha-1}\left(1-2\alpha^{-r/2}-\left|\frac{\alpha-1}{\alpha\re^{it}-1}\right|\right)_+.\label{Re>}
\end{align}
In the last steps, we used a relation $\Re z\geq -|z|$, $z\in\mathbf{C}$. Applying
\[ 1-\frac{p}{\sqrt{p^2+v^2}}\geq \frac{1}{2}\frac{v^2}{p^2+v^2},\quad p> 0,\quad v\in\R,\] 
with $p=\alpha-1$ and $v=\sqrt{2\alpha (1-\cos t)}$, and recalling \eqref{cos}, we get that
\begin{align*}
1-\left|\frac{\alpha-1}{\alpha\re^{it}-1}\right|&\geq \frac{1}{2}\frac{2\alpha (1-\cos t)}{(\alpha-1)^2+2\alpha(1-\cos t)}\\
&\geq \frac{(1-\cos(1/r))}{(\alpha-1)^2+2(1-\cos(1/r))}\\
&\geq \frac{2}{\pi^2}\frac{1}{(r\alpha-r)^2+1}.
\end{align*}

Inserting the last lower estimate into \eqref{Re>}, in conjunction with \eqref{Re}, we obtain
\begin{align}
&\sum_{j=1}^r\frac{\alpha^j\left(1+\cos(t j)\right)-2}{j}\nonumber\\ 
&\geq \sum_{j=1}^{r}\frac{\cos(t j)-1}{j}+\frac{2}{r}\frac{\alpha^{r+1}}{\alpha-1}\left(\frac{\pi^{-2}}{1+(r\alpha-r)^2}-\alpha^{-r/2}\right)_+-4.\label{trigsum}
\end{align}
It is necessary to make a specific estimate for the remaining sum. We have
\begin{align*}
\sum_{j=1}^r\frac{1-\cos(tj)}{j}&\leq  \Re\sum_{j=1}^\infty\frac{1-\re^{itj}}{j}\alpha^{-j}+\sum_{j=1}^r\frac{1-\cos(tj)}{j}(1-\alpha^{-j})\\
&\leq \log\frac{\left|1-\alpha^{-1}\re^{it}\right|}{1-\alpha^{-1}}+2\sum_{j=1}^r\frac{1-\alpha^{-j}}{j}\\
&\leq \log\frac{\left|\alpha-\re^{it}\right|}{\alpha-1}+\min\{2r\log\alpha,\ 2\log(\re r)\}.
\end{align*}
The reasons for such estimate will reveal in the proofs of Theorems. Combining latter estimate with \eqref{trigsum} and \eqref{visa} we arrive at the assertion.\\
\end{proof}

{\textit{The common part of the proofs (Theorem~\ref{thm3.2}).} We set the proofs for the case $r<n/2$. The case $n/2\leq r\leq n$ is an easy exercise. For $0<\beta<1$, we have 
\begin{align*}
\Pr\Big(\sum_{j=r+1}^njZ_j=n\Big)&=\frac{1}{2\pi i}\int_{|z|=\beta}\exp\left\{\sum_{j=r+1}^n\frac{z^j-1}{j}\right\}\frac{dz}{z^{n+1}}\\
      &=\frac{\re^{-H_n}}{2\pi i}\int_{|z|=\beta}\exp\left\{\sum_{j=1}^r\frac{1-z^j}{j}\right\}\frac{dz}{(1-z)z^{n+1}},
\end{align*}
because, when we add $\sum_{j=n+1}^\infty z^j/j$ to $\sum_{j=r+1}^n z^j/j$, the $n$-th Taylor coefficient of the integrand function does not change. Let $s=-r\log z$ and $u=n/r$, then
\begin{align*}
\Pr\Big(\sum_{j=r+1}^njZ_j=n\Big)&=\frac{\re^{-H_n}}{2\pi i}\int_{-r\log\beta-ir\pi}^{-r\log\beta+ir\pi}\exp\left\{\sum_{j=1}^r\frac{1-\re^{-js/r}}{j}\right\}\frac{\re^{us+s/r}}{r(\re^{s/r}-1)}ds\\
&=: \frac{\re^{-H_n}}{2\pi i}\int_{-r\log\beta-ir\pi}^{-r\log\beta+ir\pi} L(s)\re^{us}ds.
\end{align*}
Notice that $\operatorname{Res}\limits_{s=0}L(s)\re^{us}=1$. Assuming $\alpha>1$ we apply the residue theorem to get
\begin{align}
\Pr\Big(\sum_{j=r+1}^njZ_j=n\Big)&=\re^{-H_n}\left(1+\frac{1}{2\pi i}\int_{-r\log\alpha-ir\pi}^{-r\log\alpha+ir\pi} L(s)\re^{us}ds\right)\nonumber\\
&=:\re^{-H_n}\left(1+R\right).\label{PR}
\end{align}
So, it remains to estimate $R$. Firstly, we integrate by parts twice (with respect to $d\re^{us}$ and to $d\re^{(u-1)s}$). Note that $L'(s)=-\frac{\re^{-s}}{r(\re^{s/r}-1)}L(s),$
\begin{align}
R&=\frac{1}{2\pi iu}\int_{-r\log\alpha-ir\pi}^{-r\log\alpha+ir\pi}L(s)d\re^{us}\nonumber\\
&=\frac{1}{2\pi in(u-1)}\int_{-r\log\alpha-ir\pi}^{-r\log\alpha+ir\pi}\frac{L(s)}{\re^{s/r}-1}d\re^{(u-1)s}\nonumber\\
&=\frac{1}{2\pi in(n-r)}\int_{-r\log\alpha-ir\pi}^{-r\log\alpha+ir\pi}\frac{\re^{s/r}+\re^{-s}}{(\re^{s/r}-1)^2}L(s)\re^{(u-1)s}ds.\label{R}
\end{align}
Secondly, using abbreviations 
\begin{equation}
E(r,\alpha):=-\frac{2}{r}\frac{\alpha^{r+1}}{\alpha-1}\left(\frac{\pi^{-2}}{1+(r\alpha-r)^2}-\alpha^{-r/2} \right)_+ +\min\{2r\log\alpha, 2\log(\re r)\}\label{Eev}
\end{equation}
and
\[
J(\tau,\alpha):=\frac{1}{(\alpha-1)^2+\alpha\left(2\tau/(\pi r)\right)^2},\ \tau\in\R,
\]
we apply \eqref{cos} and Lemma~\ref{trig} to~\eqref{R}:
\begin{align}
|R|&\leq \frac{4\alpha^{2r-n+2}}{\pi n^2r}\int_0^{r\pi}\frac{1}{\left|\re^{i\tau/r}-\alpha \right|^3}\exp\left\{\sum_{j=1}^r\frac{1-\alpha^j\cos(\tau j/r)}{j}\right\}d\tau\nonumber \\
&\leq  \frac{4\re^{4}\alpha^{2r-n+2}}{\pi n^2r(\alpha-1)}\exp\left\{\sum_{j=1}^r\frac{\alpha^j-1}{j}+E(r,\alpha)\right\}\int_{1}^{r\pi}J(\tau, \alpha)d\tau\nonumber\\
&\quad\ +\frac{4\re \alpha^{2r-n+2}}{\pi n^2r}\exp\left\{-\frac{\alpha(\alpha^r-1)}{2r(\alpha-1)}\right\}\int_{0}^{1}J(\tau, \alpha)^{3/2}d\tau,\nonumber
\end{align}
because $\left|\re^{s/r}+\re^{-s}\right|\leq 2\alpha^{{r}}$, and the additional factor $1/|\re^{i\tau/r}-\alpha|$ came from $|L(s)|$. Since $$\int_{1}^{r\pi}J(\tau,\alpha)d\tau=\frac{\pi r}{2\sqrt{\alpha}(\alpha-1)}\arctan\left(\frac{2\sqrt{\alpha}\tau}{\pi r(\alpha-1)}\right)\Big|_1^{r\pi}\leq \frac{\pi^2r}{4\sqrt{\alpha}(\alpha-1)}$$
and $$\int_{0}^{1}J(\tau,\alpha)^{3/2}d\tau\leq \left(\alpha-1\right)^{-3},$$ 
we obtain
\begin{align}\label{RR}
|R| &\leq  \frac{\pi\re^4\alpha^{2r-n+3/2}}{n^2(\alpha-1)^2}\exp\left\{\sum_{j=1}^r\frac{\alpha^j-1}{j}+E(r,\alpha)\right\}\nonumber\\
&\quad\ +\frac{4\re\alpha^{2r-n+2}}{\pi n^2r(\alpha-1)^3}\exp\left\{-\frac{\alpha(\alpha^r-1)}{2r(\alpha-1)}\right\}.
\end{align}

Hereafter, we divide the argument into two parts. In the first part we take $\alpha=\re^{\xi/r}$ (Theorem~\ref{theorem3.3}) and in the second $\alpha=x$ (Theorem~\ref{theorem3.4}), where $\xi=\xi(n/r)$ is defined in Lemma~\ref{xi} and $x=x(n,r)$ in Lemma~\ref{x}.\qed

\begin{proof}[Proof of Theorem~\ref{theorem3.3}] Let $\alpha=\re^{\xi/r}$. Recall that $\log n\leq r<n/2$ and $\log u\leq \xi\leq 2\log u$. Note, referring to the equation $\re^\xi-1=u\xi$, we have $\xi>1$.  Now, we need the following estimates
\begin{align*}
I(\xi)&=\int_0^\xi\frac{1}{t} d(\re^t-t-1)\geq \frac{\re^{t}-t-1}{t}\Big|_0^\xi = u-1,
\end{align*}
\begin{equation}\label{T}
0 \leq T(\xi)\leq  \frac{2u\xi}{3r}+\frac{1}{r}
\end{equation}
(this follows from \eqref{Tz3}; note $\xi\leq 2r$), and
\begin{align*}
& E(r,\re^{\xi/r})-\min\{2\xi, 2\log(\re r)\}\nonumber\\
&= -\frac{2}{r}\frac{\re^{\xi+\xi/r}}{\re^{\xi/r}-1}\left(\frac{\pi^{-2}}{1+(r\re^{\xi/r}-r)^2}-\re^{-\xi/2}\right)_+\\
&= \frac{-2\re^{\xi}\pi^{-2}}{\xi+\xi(r\re^{\xi/r}-r)^2}\frac{\xi\re^{\xi/r}}{r(\re^{\xi/r}-1)}\left(1-\frac{\pi^2+\pi^2(r\re^{\xi/r}-r)^2}{\re^{\xi/2}}\right)_+.
\end{align*}
Now, applying estimates $\xi\leq 2r$, $(\xi/r)\re^{\xi/r}/(\re^{\xi/r}-1)=1+O_+(\xi/r)$ (see \eqref{Bern}), and $r\re^{\xi/r}-r=\xi(1+O_+(\xi/r))$, where $O_+$ is an estimate of the positive quantity, we obtain
\begin{align}
& E(r,\re^{\xi/r})-\min\{2\xi, 2\log(\re r)\}\nonumber\\
&= \frac{-2(u+1/\xi)\pi^{-2}}{1+\xi^2(1+O_+(\xi/r))}\frac{\xi\re^{\xi/r}}{r(\re^{\xi/r}-1)}\left(1+O\left(\frac{(\log u)^{3/2}}{\sqrt{u}}\right)\right) \label{E1}\\
&= \frac{-2(u+1/\xi)\pi^{-2}}{1+\xi^2(1+O_+(\xi/r))}\left(1+O\left(\frac{\log u}{r}+\frac{(\log u)^{3/2}}{\sqrt{u}}\right)\right)\label{E2}
\end{align}
Certainly, there exists constants $c_1, c_2>0$ such that $I(\xi)+T(\xi)+E(r,\re^{\xi/r})\geq c_1 u$ for all $u>c_2$. Recalling \eqref{IT} and applying latter estimate to $\eqref{RR}$, we conclude that
\begin{align}\label{RRR}
R & \ll \frac{\re^{2\xi-u\xi}}{(u\xi)^2}\exp\left\{\sum_{j=1}^r\frac{\re^{j\xi/r}-1}{j}+E(r,\re^{\xi/r}) \right\}\nonumber\\ 
& \ll \exp\left\{-u\xi+I(\xi)+T(\xi)+E(r,\re^{\xi/r})\right\}.
\end{align}
To obtain \eqref{thm122}, we first observe that combining \eqref{P-ell}, \eqref{PR} and \eqref{RRR} we get
\[
\kappa(n,r)=\re^{-H_r}+O\left( \frac{1}{r}\exp\left\{-u\xi+I(\xi)+T(\xi)+E(r,\re^{\xi/r})\right\}\right).
\]
Here, we apply Lemma~\ref{rho} and formula \eqref{deriv} to obtain
\begin{equation}\label{almostepsilon}
\kappa(n,r)=\re^{-H_r}+O\left( \frac{\varrho(u)}{r}\exp\left\{T(\xi)+E(r,\re^{\xi/r})+\log u\right\}\right).
\end{equation}
It is left to evaluate $T(\xi)+E(r,\re^{\xi/r})+\log u$. For that we use \eqref{T}, \eqref{E1}-\eqref{E2}, and then we apply \eqref{xi asymp}. Recall that $\log n\leq r\leq (\log n)^{3+\varepsilon}$, we have
\begin{align}
& T(\xi)+E(r,\re^{\xi/r})+\log u\nonumber\\
&\leq \frac{2u\xi}{3r}+\frac{1}{r}+2\xi-\frac{2(u+1/\xi)\pi^{-2}}{1+\xi^2(1+O_+(\xi/r))}\left(1+O\left(\frac{\log u}{r}+\frac{(\log u)^{3/2}}{\sqrt{u}}\right)\right)+\log u\nonumber\\
&=\frac{2u\xi}{3r}\left(1-\frac{r}{\xi^3}\cdot\frac{3}{(\pi/\xi)^2+\pi^2(1+O_+(\xi/r))}\left(1+O\left(\frac{\log u}{r}+\frac{(\log u)^{3/2}}{\sqrt{u}}\right)\right) \right)\nonumber\\
&\quad\ +O\left(\xi\right)\nonumber\\
&\leq \frac{u\log u}{r}\label{lastepsitimate}
\end{align}
if $u$ is sufficiently large, because $(\log u)/r$ (in the error estimate) came from the factor in \eqref{E1} that is greater than 1. When $u$ is bounded, equation \eqref{thm122} is trivial. Applying \eqref{lastepsitimate} to \eqref{almostepsilon}, we prove \eqref{thm122}. 

Now we turn to assertion \eqref{thm121}, recalling that $(\log n)^{3+\varepsilon}\leq r< n$, $\varepsilon> 0$, and $\delta> 0$. Let $n_0=n_0(\varepsilon, \delta)$ and $u_0=u_0(\varepsilon, \delta)$ be such sufficiently large positive constants, depending on parameters $\varepsilon$ and $\delta$, that if $n\geq n_0$ and $u\geq u_0$, using \eqref{E2}, \eqref{T}, and Lemma~\ref{xi}, we obtain
\begin{align*}
&E(r,\re^{\xi/r})+2\xi+T(\xi)\\
&\leq -\frac{2u}{\pi^2\xi^2}\left(1+O\left(\frac{\log u}{r}+\frac{1}{(\log u)^2}\right)\right)+4\xi+ \frac{2u\xi}{3r}+\frac{1}{r}\\
&=  -\frac{2u}{\pi^2\xi^2}\left(1+O\left(\frac{1}{(\log u)^2}+\frac{(\log u)^3}{r}\right)\right)\\
&=  -\frac{2u}{\pi^2 (\log u)^2}\left(1+O\left(\frac{\log\log(u+2)}{\log u}+\frac{1}{(\log n)^\varepsilon}\right)\right)\\
&\leq -\frac{2u}{\pi^2(\log (1+u))^2}\left(1-\delta\right).
\end{align*} 
Combining the latter estimate with \eqref{RRR}, \eqref{PR}, and Lemma~\ref{rho} we prove assertion~\eqref{thm121} for the case $n\geq n_0$ and $u\geq u_0$. In case when $n\leq n_0$ or $u\leq u_0$, assertion~\eqref{thm121} follows trivially from \eqref{PR}, \eqref{RRR}, \eqref{E2} and Lemma~\ref{rho}.
\end{proof}

\begin{proof}[Proof of Theorem~\ref{theorem3.4}] Let $\alpha=x$. Recall that $u>2$ and $r\geq 5$. Throughout the proof, we apply inequalities $\re^{(\log u)/r}\leq x\leq \re^{(2\log u)/r}$ and equation 
	\[
	\frac{x^{r+1}-x}{x-1}=n,
	\]   
	which come from Lemma~\ref{x}.
	Recall \eqref{Eev}. The proof starts with the observation that
	\begin{align*}
	E(r,x)-\min\{2r\log x, 2\log(\re r)\}&\leq -2u\left(\frac{\pi^{-2}}{1+(rx-r)^2}-x^{-r/2}\right)_+\\
	&\leq \frac{-u^{1-4/r}}{4\pi^2(\log u)^2}\left(1-\frac{\pi^2}{x^{r/2}}-\frac{\pi^2(rx-r)^2}{x^{r/2}}\right)_+\\
	&\leq \frac{-u^{1-4/r}}{4\pi^2(\log u)^2}\left(1-\frac{\pi^2}{\sqrt{u}}-\frac{\pi^2 x^{3r/2+2}}{u^2}\right)_+,
	\end{align*}
	where the second inequality is obtained using estimate $rx-r\leq r\re^{(2\log u)/r}-r\leq 2(\log u)u^{2/r}$, and the third -- inequalities $x\geq \re^{(\log u)/r}$ and $(x-1)/x^r\leq x/n.$
	Now, we apply the latter estimate and $x^2/(nx-n)^2\ll 1$ to \eqref{RR}, and thus obtain
	\begin{align*}
	R&\ll  x^{2r-n}\exp\left\{\sum_{j=1}^r\frac{x^j-1}{j}-\frac{u^{1-4/r}}{4\pi^2(\log u)^2}\left(1-\frac{\pi^2}{\sqrt{u}}-\frac{\pi^2 x^{3r/2+2}}{u^2}\right)_+\right\}.
	\end{align*}
	Applying Theorem~\ref{thm1} with Corollary~\ref{2cor}, and inequalities $r^2<\sum_{j=1}^rjx^j\leq rn$, which follow from $\sum_{j=1}^rj\leq \sum_{j=1}^rjx^j\leq r\sum_{j=1}^rx^j$, it follows that
	\begin{align}
	R &\ll \sqrt{u}x^{2r}\nu(n,r)\exp\left\{-\frac{u^{1-4/r}}{4\pi^2(\log u)^2}\left(1-\frac{\pi^2}{\sqrt{u}}-\frac{\pi^2 x^{3r/2+2}}{u^2}\right)_+\right\}\label{penultimateR}\\
	&\ll \nu(n,r)\exp\left\{-\frac{u^{1-4/r}}{4\pi^2(\log u)^2}\left(1-\frac{\pi^2}{\sqrt{u}}-\frac{\pi^2 x^{3r/2+2}}{u^2}\right)_++5\log u\right\}.\label{lastR}
	\end{align}
	Now, formula \eqref{132} follows simply from \eqref{penultimateR} when we apply \eqref{wedge}. It remains to prove the formula \eqref{131}. Let $u_1=u_1(\varepsilon)$ be such sufficiently large positive constant depending on parameter $\varepsilon>0$ that if $u\geq u_1$  applying \eqref{wedge} to \eqref{lastR} it follows that
	\begin{align}
	R &\ll  \nu(n,r)\exp\left\{-\frac{u^{1-4/r}}{4\pi^2(\log u)^2}\left(1+O\left((\log u)^3u^{2/r-1/2}\right)\right)\right\}\nonumber\\
	&\ll_\varepsilon  \nu(n,r)\exp\left\{-\frac{u^{1-4/r}(1-\varepsilon)}{4\pi^2(\log(u+1))^2}\right\}.\nonumber
	\end{align}
	Actually, this estimate is also correct in the case of $u\leq u_1$. Recalling~\eqref{PR} and \eqref{P-ell} we finish the proof of \eqref{131}.  
\end{proof}

\begin{proof}[Proof of Corollary~\ref{corollary3.5}]
	Using Theorem~\ref{Theorem 1} we deduce Lemma~\ref{very small} which combined with Theorem~\ref{theorem3.4} gives the assertion.
	\end{proof}

\section{Proof of Theorem~\ref{theorem3.2}}

\bigskip

\textbf{Theorem \ref{theorem3.2}.} \textit{Let $u=n/r$. If $\sqrt{n\log n}\leq r<n$, then 
	\begin{equation*}
	\kappa(n,r)=\re^{-H_r+\gamma} \omega(u)+O\left(\frac{R(u)u^{3/2}\log^2 (u+1)}{r^2}\right),
	\end{equation*}
	where the function $R(u)$ is described by Definition~\ref{Rdefinition} and Lemmas~\ref{omega,R} and~\ref{Rneeded}.}
\begin{proof}
	Throughout the argument, we will use the following abbreviations 
	$$\xi_0(u):=\left\{
	\begin{array}{ll}
	\Re\zeta_0(u) & (1< v_1\leq u)\\
	\xi_0(v_1)    & (1\leq u\leq v_1)
	\end{array}
	\right.$$
	(see Lemma~\ref{zeta} and Lemma~\ref{Robert}),
	$$u_0:=\max\{v_0, v_1, 3\}$$ 
	(for $v_0$ see \eqref{R estimate}), and
	$$N(n,r):=\re^{H_r}\kappa(n,r).$$
	Proposition \ref{EM} gives the result for $u\leq u_0$, note that $u_0$ is fixed. Given proof will concern the region $u_0\leq u\leq \sqrt{n/\log n}$. We split it into a few steps.
	
	\smallskip
	
	\textbf{1.}  As usual, let $[z^n] g(z)$ be the $n$th power series coefficient of a function $g(z)$. Then
	\[
	N(n,r)=\re^{H_r}[z^n]\exp\bigg\{\sum_{r<j\leq n}\frac{z^j}{j}\bigg\}=
	[z^n]\frac{1}{1-z}\exp\big\{G(z)\big\},
	\]
	where
	\begin{align*}
	G(z)&:=\sum_{j=1}^r\frac{1-z^j}{j}\\
	&=\int_0^{r\log z}\frac{1-\re^t}{t}dt+\int_0^{r\log z}\frac{1-\re^t}{t}\left(\frac{\frac{t}{r}\re^{\frac{t}{r}}}{\re^{\frac{t}{r}}-1}-1\right)dt\\
	&=
	-I(r\log z)-T(r\log z).
	\end{align*}
	By the Cauchy formula for $0<\alpha<1<\alpha'$,
	\begin{align*}
	N(n,r)&=\frac{1}{2\pi i}\int_{|z|=\alpha}\frac{\re^{G(z)} dz}{(1-z)z^{n+1}}=1+\frac{1}{2\pi i}\int_{|z|=\alpha'}\frac{\re^{G(z)} dz}{(1-z)z^{n+1}}\\
	&=
	1+\frac{1}{2\pi n i}\int_{|z|=\alpha'}\frac{\re^{G(z)} dz}{(1-z)^2z^{n-r}}
	\end{align*}
	because of the equality $\frac{d}{dz}\big(\exp\{G(z)\}/(1-z)\big)=z^r\exp\{G(z)\}/(1-z)^{2}$.
	We further have to shift the integration contour to a path going closely to the saddle point. Thus, after a substitution $z=\re^{-s/r}$, where $s\in\Delta:=\{s=-\xi_0(u)+i \tau:\;  -\pi r <\tau\leq \pi r\}$, by virtue of Lemma \ref{hatrho}, we obtain
	\begin{align}
	N(n,r)&=
	1+\frac{1}{2\pi n r i}\int_{\Delta} \frac{\re^{(u-1)s-I(-s)-s/r} ds}{(1-\re^{-s/r})^2}\nonumber\\
	&\quad\
	+\frac{1}{2\pi nr i}\int_{\Delta} \frac{\re^{(u-1)s-I(-s)-s/r}}{(1-\re^{-s/r})^2}
	\big(\re^{-T(-s)}-1\big) ds \nonumber\\
	&=
	1+\frac{\re^\gamma}{2\pi n r i}\int_{\Delta}\frac{\re^{(u-1)s+s/r}ds}{\hat\varrho(s)(\re^{s/r}-1)^2}\nonumber\\
	&\quad\
	+\frac{\re^\gamma}{2\pi n r i}\int_{\Delta} \frac{\re^{(u-1)s+s/r}\big(\re^{-T(-s)}-1\big) ds}{\hat\varrho(s)(\re^{s/r}-1)^2}
	\nonumber\\
	&=: 1+J_0+J_1.
	\label{J0R}
	\end{align}
	
	We can hardly expect a global trick to overcome the tiresome calculations of the appearing integrals; therefore, from now on we begin the divide and conquer quest.
	To follow the required sharpness in estimates, it is worth  to have in mind that, by Lemmas \ref{rho}, \ref{zeta} and \ref{omega,R},
	\begin{equation}
	R(u)= \exp\Big\{-u\xi_0(u)+u +O\left(\frac{u}{\xi(u)}\right)\Big\}, \quad u\geq u_0.
	\label{Rugrubus}
	\end{equation}
	Afterwards, we will use the abbreviations $\xi_0=\xi_0(u)$ and $\xi=\xi(u)$, where $u=n/r\leq \sqrt{n/\log n}$ and $n$ is large.
	Keep in mind that $\re^{\xi_0}=\re^{\xi}(1+O(\xi^{-2}))=u\log u+O(u)=O(r)$ if $u\geq u_0$.
	
	\medskip

	\textbf{2.} This step is devoted to integral $J_0$. By Lemma  \ref{hatomega}, we may start with
	\begin{align}
	J_0&=\frac{\re^\gamma}{2\pi nr i}\int_{\Delta} \frac{\re^{(u-1)s+s/r}s(1+\hat\omega(s))ds}{(\re^{s/r}-1)^2}\nonumber\\
	&=
	\frac{\re^\gamma}{2\pi u i}\int_{\Delta} \re^{(u-1)s} \frac{(1+\hat\omega(s))}{s} ds\nonumber\\
	&\quad\ +
	\frac{\re^\gamma}{2\pi u i}\int_{\Delta} \re^{(u-1)s} \bigg(\frac{\re^{s/r}(s/r)^2}{(\re^{s/r}-1)^2}-1\bigg) \frac{(1+\hat\omega(s))}{s}ds\nonumber\\
	&=:
	J_{01}+J_{02}.
	\label{J01J02}
	\end{align}
	The goal is to extract from $J_{01}$ the inverse Laplace transform
	\[
	\omega(u)=\int_{(\beta)} \re^{us} \hat\omega(s) ds,
	\]
	defined for arbitrary $\beta>0$. Here, and in the sequel, the notation $(b)$, $b\in\R$, under integral denotes the integration line from  $b-i\infty$ to $b+i\infty$.
	
	From Lemmas \ref{IsJs} and \ref{hatomega}, we have
	$\hat\omega(s)=\re^{J(s)}-1$ and, consequently, $\hat\omega(-\xi_0\pm ir\pi)\ll\re^{\xi}/r$. This and relation (\ref{tapat}) imply
	\[
	J_{01}=
	-\frac{\re^\gamma}{2\pi u i}\int_{\Delta} \re^{us} d\hat\omega(s)
	=
	O\left(\frac{\re^{-u\xi_0+\xi}}{r}+\frac{\re^\gamma}{2\pi  i}\int_{\Delta} \re^{us}\hat\omega(s) ds\right).
	\]
	Similarly, using again Lemma \ref{hatomega} and  $J(s)=\re^{-s}/s+O\left(\re^{\xi}/|s|^2\right)$ if $s=-\xi_0+i\tau$, where $\pi r\leq |\tau|\leq T$, we obtain
	\begin{align*}
	& \lim_{T\to\infty}\frac{\re^\gamma}{2\pi i}\int_{\substack{\Re s=-\xi_0 \\ r\pi\leq |\tau|\leq T}}\re^{us}\hat{\omega}(s) ds\\
	&=
	\lim_{T\to\infty}\left[ \frac{\re^\gamma}{2\pi i}\int_{\substack{\Re s=-\xi_0 \\ r\pi\leq |\tau|\leq T}}\frac{\re^{(u-1)s}}{s}ds+O\left(\re^{2\xi-u\xi_0}\int_{\substack{\Re s=-\xi_0 \\ r\pi\leq |\tau|\leq T}}\frac{1}{|s|^2}|ds|\right)\right]\\
	&\ll \frac{\re^{-u\xi_0+2\xi}}{r}.
	\end{align*}
	Collecting the last two estimates, we complete estimation of $J_{01}$. Indeed, for a fixed $\beta>0$, by the residue theorem, we obtain
	\begin{align}
	J_{01}
	&=
	\frac{\re^\gamma}{2\pi  i}\int_{(-\xi_0)} \re^{us}\hat\omega(s) ds + O\left(\frac{\re^{-u\xi_0+2\xi}}{r}\right)\nonumber\\
	&=
	-1+\frac{\re^\gamma}{2\pi  i}\int_{(\beta)} \re^{us}\hat\omega(s) ds + O\left(\frac{\re^{-u\xi_0+2\xi}}{r}\right)\nonumber\\
	&=
	-1+\re^\gamma\omega(u)+O\left(\frac{\re^{-u\xi_0+2\xi}}{r}\right).
	\label{J01Final}
	\end{align}
	By (\ref{Rugrubus}), the remainder here is better than indicated in Theorem \ref{theorem3.2}.

	The remaining  integral in (\ref{J01J02}) is
	\[
	J_{02}= \frac{\re^\gamma}{2\pi u(u-1) i}\int_{\Delta}  W\Big(\frac{s}{r}\Big)\frac{(1+\hat\omega(s))}{s}d\re^{(u-1)s},
	\]
	where
	\begin{equation}
	W(v)=\frac{\re^{v}v^2}{(\re^{v}-1)^2}-1\ll v^2
	\label{Wv}
	\end{equation}
	if $v=s/r$ and $s\in \Delta$. Moreover, $W'(v)\ll v$ in the same region. Hence,
	\[
	\bigg(\frac{\hat\omega(s)+1}{s} W\Big(\frac{s}{r}\Big)\bigg)'\ll \frac{|s||\hat\omega(s)'|}{r^2}+\frac{|\hat\omega(s)+1|}{r^2}\ll \re^\xi\frac{|\hat\omega(s)+1|}{r^2}
	\]
	by (\ref{tapat}).
	For  $\hat\omega(s)+1$,  where $s=-\xi_0+i\tau\in\Delta$, we further  apply Lemma \ref{Robert} if $|\tau|\leq \re^\xi$ and $\hat\omega(s)+1 = O\left(1\right)$ otherwise.  The latter stems from Lemma \ref{hatomega}. So, we obtain
	\begin{align*}
	J_{02}&=
	\frac{\re^\gamma}{2\pi u (u-1) i}\int_{\Delta}\re^{(u-1)s} \bigg(  W\Big(\frac{s}{r}\Big)\frac{(1+\hat\omega(s))}{s}\bigg)' d s+
	O\left(\frac{\re^{-u\xi_0+\xi}}{nu }\right)\\
	&\ll
	\frac{\re^{-u\xi_0+2\xi}}{n^2}\int_{\Delta} |\hat\omega(s)+1| |d s|+
	\frac{\re^{-u\xi_0+\xi}}{nu }\\
	&\ll
	\frac{\re^{3\xi}R(u)\sqrt u}{n^2 }+ \frac{\re^{-u\xi_0+2\xi}}{nu }\ll \frac{R(u)\sqrt u \log^2 u}{r }.
	\end{align*}
	
	Returning to (\ref{J01J02}) we obtain an asymptotic formula for integral $J_0$. Namely, we have
	\begin{equation}
	J_0=-1+\re^\gamma\omega(u)+O\left(\frac{R(u)\sqrt u\log^2 u}{r}\right).
	\label{J0Final}
	\end{equation}
	
	\medskip

	\textbf{3.} Estimation of the integral $J_1$ involving function $T(s)$ is more subtle.  In the notation above,
	\begin{align*}
	J_1&=
	\frac{\re^\gamma}{2\pi i n(u-1)}\int_{\Delta}\left(\hat{\omega}(s)+1\right) \frac{(s/r)\re^{s/r}}{(\re^{s/r}-1)^2}\big(\re^{-T(-s)}-1\big) d\re^{(u-1)s}\\
	&=
	-\frac{\re^\gamma}{2\pi in(u-1)}\int_{\Delta} \re^{(u-1)s}\left[\left(\hat{\omega}(s)+1\right)\Big(\frac{r}{s}W\Big(\frac{s}{r}\Big)+\frac{r}{s}\Big) \big(\re^{-T(-s)}-1\big)\right]' ds \\
	&\quad\
	+O\left(\frac{\re^{-u\xi_0+\xi}}{nu}\right)\\
	&=:-\frac{\re^\gamma}{2\pi in(u-1)}\int_{\Delta} \re^{(u-1)s}\left[\left(\hat{\omega}(s)+1\right)\Big(\frac{r}{s}W\Big(\frac{s}{r}\Big)\Big) \big(\re^{-T(-s)}-1\big)\right]' ds \\
	&\quad\
	+J_{12}+O\left(\frac{\re^{-u\xi_0+\xi}}{nu}\right)\\
	&=:
	J_{11}+J_{12}+O\left(\frac{\re^{-u\xi_0+\xi}}{nu}\right).
	\end{align*}
	The remainder here has been estimated using  Lemma \ref{Tz}  and (\ref{Wv}). Observe that integral $J_{11}$ becomes $J_{12}$ if we substitute $W(s/r)$ by the unit.
	
	Taking into account identity (\ref{tapat}), we have
	\begin{align}
	&\left[\left(\hat{\omega}(s)+1\right)\Big(\frac{r}{s}W\Big(\frac{s}{r}\Big)\Big) \big(\re^{-T(-s)}-1\big)\right]'\nonumber\\
	&=
	-\frac{r(\hat{\omega}(s)+1)}{s}U(s)W\Big(\frac{s}{r}\Big)+\frac{\hat{\omega}(s)+1}{s}W'\Big(\frac{s}{r}\Big) \big(\re^{-T(-s)}-1\big).
	\label{suWrs}
	\end{align}
	Here
	\begin{align}
	U(s)&:=\big(\re^{-T(-s)}-1\big)\frac{\re^{-s}+1}{s}-  T'(-s)\re^{-T(-s)}\nonumber\\
	&=
	\frac{\re^{-s}}{s}\bigg(\re^{-T(-s)}\frac{s/r}{\re^{s/r}-1}-1\bigg)
	+\frac{1}{s}\bigg(2\re^{-T(-s)}-1-\re^{-T(-s)} \frac{s/r}{\re^{s/r}-1}\bigg)\nonumber\\
	&=
	\frac{\re^{-s}}{s}\bigg(-\frac{s}{r}+O\left(\frac{\re^\xi}{r}\right)+O\left(\frac{|s|^2}{r^2}\right)\bigg)+O\bigg(\frac{\re^\xi}{r|s|}+\frac{|s|}{r^2}\bigg)\nonumber\\
	&=
	-\frac{\re^{-s}}{r}+O\left(\frac{\re^{2\xi}}{|s|r}+\frac{|s|\re^{\xi}}{r^2}\right),
	\quad s\in \Delta,
	\label{Us}
	\end{align}
	by Lemma \ref{Tz}. Hence and from (\ref{Wv}) we obtain
	\[
	s^{-1}U(s)W(s/r)\ll\re^{\xi}r^{-2}
	\]
	if $s\in\Delta$.
	
	Applying again a bound $\hat\omega(s)+1=O\left(1\right)$ if $\tau>\re^\xi$ and Lemma \ref{Robert} otherwise and having in mind  that $\re^{-T(-s)}-1=O\left(1\right)$ if $s\in\Delta$, from the last relation and (\ref{suWrs}), we derive
	\begin{align*}
	J_{11}&=
	\frac{\re^\gamma }{2\pi iu(u-1)}\int_{\Delta} \re^{(u-1)s}\frac{(\hat{\omega}(s)+1)}{s} U(s)W\Big(\frac{s}{r}\Big) ds+O\left(\frac{R(u)}{r}\right)\\
	&\ll
	\frac{R(u)\sqrt u}{u^2 r^2}\re^{3\xi}+ \frac{1}{u^2 r}\re^{-u\xi_0+2\xi}+\frac{R(u)}{r}\\
	&\ll
	\frac{R(u)\sqrt u\log^2 u}{r}.
	\end{align*}
	
	Similarly, by (\ref{Us}),
	\begin{align*}
	J_{12}&=
	\frac{\re^\gamma }{2\pi iu(u-1)}\int_{\Delta} \re^{(u-1)s}\frac{(\hat{\omega}(s)+1)}{s} U(s) ds\\
	&=
	-\frac{\re^\gamma }{2\pi i n(u-1)}\int_{\Delta}\frac{(\hat{\omega}(s)+1)}{s}\re^{(u-2)s} ds\\
	&\quad\ +
	O\left(\frac{\re^{2\xi}R(u)\sqrt u}{u^2}\int_\Delta {\mathbf 1}\{|\Im s|\leq \re^\xi\}  \Big(\frac{\re^\xi}{r|s|^2}+\frac{1}{r^2}\Big) |ds|\right)\\
	&\quad\ +
	O\left(\frac{\re^{-u\xi_0+2\xi}}{u^2}\int_\Delta {\mathbf 1}\{\re^\xi\leq |\Im s|\leq \pi r\}  \Big(\frac{\re^\xi}{r|s|^2}+\frac{1}{r^2}\Big) |ds|\right)\\
	&=
	-\frac{\re^\gamma }{2\pi i n(u-1)}\int_{\Delta}\frac{(\hat{\omega}(s)+1)}{s}\re^{(u-2)s} ds+O\left(\frac{u^{3/2}R(u)\log^2 u}{r}\right).
	\end{align*}
	We have dealt with a such type integral! Apart from a factor and a substitution $u\mapsto u-1$, the last integral is just $J_{01}$ in (\ref{J01J02}) of step 2. By virtue of (\ref{J01Final}), this yields
	\begin{align*}
	J_{12}
	&=
	\frac{1}{n}\bigg(1-\re^\gamma\omega(u-1)+ O\left(\frac{\re^{-u\xi_0+3\xi}}{r}\right)\bigg) +O\left(\frac{u^{3/2}R(u)\log^2 u}{r}\right)\\
	&\ll
	\frac{u^{3/2}R(u)\log^2 u}{r}.
	\end{align*}
	Collecting the estimates, we obtain
	\[
	J_{1}=  J_{11}+J_{12}\ll \frac{u^{3/2}R(u)\log^2 u}{r}.
	\]
	
	Finally, combining estimates (\ref{J0R}), (\ref{J0Final}) and the last one, we obtain
	\[
	N(n,r)=1+J_0+J_1=\re^\gamma \omega(u)+O\left(\frac{u^{3/2}R(u)\log^2 u}{r}\right).
	\]
	
\end{proof}

\newpage
\section{Proof of Theorem~\ref{EMRP}}\label{section5.8}

\bigskip

\textbf{Theorem~\ref{EMRP}.}  \textit{Let $u=n/r$. If $\sqrt{n\log n}\leq r\leq n$, then
	\begin{equation*}
	d_{TV}(n,r)=H(u)\bigg(1+O\left(\frac{u^{3/2}\log^2 (u+1)}{r}\right)\bigg),
	\end{equation*}}
\noindent where 
\[
H(u)=\frac{1}{2}\int_0^\infty\big|\omega(u-v)-{\re}^{-\gamma}\big| \varrho(v) dv +\frac{\varrho(u)}{2}.
\]

\medskip

At first, we list some facts and lemmas needed for the proof. 

\begin{lemma}\label{dtvformula} For the random vectors $\bar{k}_r(\sigma)=(k_1(\sigma),\ldots,k_r(\sigma))$ and $\bar{Z}_r=(Z_1,\ldots,Z_r)$ described in Chapter~\ref{Chapter4}, we have 
	\begin{align}
	d_{TV}(n,r)&=\sup_{V\subset\N^r_0}\left|\frac{\#\{\sigma : \bar{k}_r(\sigma)\in V\}}{n!}-\Pr\left(\bar{Z}_r\in V\right)\right|\nonumber\\
	&=\frac{1}{2}\sum_{m=0}^\infty \Pr\Big(\sum_{j=1}^rjZ_j=m\Big)\left|\re^{H_n}\Pr\Big(\sum_{j=r+1}^n jZ_j=n-m\Big)-1\right|\label{dtv(n,r)andP}\\
	&=\frac{1}{2}\sum_{m=0}^\infty\nu(m,r)\left|\kappa(n-m,r)-\re^{-H_r}\right|\label{dtv(n,r)}	
	\end{align}
	if $1\leq r\leq n$.
\end{lemma}
\begin{proof}
	Here we recall some calculations from \cite[pp.~67-69]{ABT}. Let $\bar{s}_r=(s_1,\ldots,s_r)\in\N^r_0$, $\bar{s}=(s_1,\ldots,s_n)\in\N^n_0$, $\bar{Z}:=\bar{Z}_n$, and $\ell(\bar{s})=1s_1+2s_2+\ldots+ns_n$, then
	
	\begin{align*}
		\frac{\#\{\sigma : \bar{k}_r(\sigma)=\bar{s}_r\}}{n!}&=\prod_{j=1}^r\frac{1}{j^{s_j}s_j!}\Bigg(\sum_{\substack{\forall (s_{r+1},\ldots,s_n)\in\N_0^{n-r} \\ \ell(\bar{s})=n}}\prod_{j=r+1}^n\frac{1}{j^{s_j}s_j!}\Bigg)\\
		&= \Pr\left(\bar{Z}_r=\bar{s}_r\ |\ \ell(\bar{Z})=n\right).
	\end{align*}
	
	Therefore, 
	\begin{equation}\label{dtv}
	d_{TV}(n,r)=\sup_{V\subset\N^r_0}\left|\Pr\left(\bar{Z}_r\in V\ |\ \ell(\bar{Z})=n\right)-\Pr\left(\bar{Z}_r\in V\right)\right|.
	\end{equation}
	
	Defining 
	\[
	V^>=\left\{\bar{s}_r\in\N^r_0\ :\ \Pr\left(\bar{Z}_r=\bar{s}_r\ |\ \ell(\bar{Z})=n\right)>\Pr\left(\bar{Z}_r=\bar{s}_r\right)\right\}
	\]
	and
	\[
	V^{\geq}=\left\{\bar{s}_r\in\N^r_0\ :\ \Pr\left(\bar{Z}_r=\bar{s}_r\ |\ \ell(\bar{Z})=n\right)\geq\Pr\left(\bar{Z}_r=\bar{s}_r\right)\right\},
	\]
	it is easy to see that a set $V$ achieves the supremum in \eqref{dtv} if and only if $V^{>}\subset V\subset V^{\geq}$; in particular 
	\begin{align*}
		d_{TV}(n,r)&=\Pr\left(\bar{Z}_r\in V^>\ |\ \ell(\bar{Z})=n\right)-\Pr\left(\bar{Z}_r\in V^>\right)\\
		&=\sum_{\bar{s}_r\in\N_0^r}\left(\Pr\left(\bar{Z}_r=\bar{s}_r\ |\ \ell(\bar{Z})=n\right)-\Pr\left(\bar{Z}_r=\bar{s}_r\right)\right)^+.
	\end{align*}
	
	We have written $a^+$ and $a^-$ for the positive and negative parts of a real number $a$, so $a=a^+-a^-$ and $|a|=a^++a^-$. The relation 
	\[
	\sum_{\bar{s}_r\in\N_0^r}\Pr\left(\bar{Z}_r=\bar{s}_r\ |\ \ell(\bar{Z})=n\right)=1=\sum_{\bar{s}_r\in\N_0^r}\Pr\left(\bar{Z}_r=\bar{s}_r\right)
	\]
	
	implies that 
	\begin{align*}
		&\sum_{\bar{s}_r\in\N_0^r}\left(\Pr\left(\bar{Z}_r=\bar{s}_r\ |\ \ell(\bar{Z})=n\right)-\Pr\left(\bar{Z}_r=\bar{s}_r\right)\right)^+\\
		&=\sum_{\bar{s}_r\in\N_0^r}\left(\Pr\left(\bar{Z}_r=\bar{s}_r\ |\ \ell(\bar{Z})=n\right)-\Pr\left(\bar{Z}_r=\bar{s}_r\right)\right)^-\\
		&=\frac{1}{2}\sum_{\bar{s}_r\in\N_0^r}\left|\Pr\left(\bar{Z}_r=\bar{s}_r\ |\ \ell(\bar{Z})=n\right)-\Pr\left(\bar{Z}_r=\bar{s}_r\right)\right|
	\end{align*}
	
	Thus, it follows that 
	
	\begin{align*}
		2d_{TV}(n,r)&=\sum_{\bar{s}_r\in\N_0^r}\left|\Pr\left(\bar{Z}_r=\bar{s}_r\ |\ \ell(\bar{Z})=n\right)-\Pr\left(\bar{Z}_r=\bar{s}_r\right)\right|\nonumber\\
		&=\sum_{\bar{s}_r\in\N_0^r}\left|\frac{\Pr\left(\bar{Z}_r=\bar{s}_r,\ \ell(\bar{Z})=n\right)}{\Pr\left(\ell(\bar{Z})=n\right)}-\Pr\left(\bar{Z}_r=\bar{s}_r\right)\right|\nonumber\\	
		&=\sum_{m\geq0} \sum_{\substack{\bar{s}_r\in\N_0^r\\ 1s_1+\dots+rs_r=m}}\Bigg|\frac{\Pr\left(\bar{Z}_r=\bar{s}_r\right) \Pr\left(\sum_{j=r+1}^n jZ_j=n-m\right)}{\Pr\left(\ell(\bar{Z})=n\right)}\\
		&\qquad -\Pr\left(\bar{Z}_r=\bar{s}_r\right)\Bigg|\nonumber\\	
		&=\sum_{m\geq0}\Pr\Big(\sum_{j=1}^rjZ_j=m\Big)\left|\frac{ \Pr\left(\sum_{j=r+1}^n jZ_j=n-m\right)}{\Pr\left(\ell(\bar{Z})=n\right)}-1\right|.
	\end{align*}		
	
	An easy computation shows that applying \eqref{P-ell}, \eqref{P-ell2}, and the fact $\Pr\left(\ell(\bar{Z})=n\right)=\re^{-H_n}$, we arrive at the proposition.
	
\end{proof}

An accurate summation of series in   (\ref{dtv(n,r)}) defining the distance $d_{TV}(n,r)$ suffices.
For this, we will apply the well-known inequality. 

\begin{lemma} \label{KH-ineq} Let $r,a,b\in\N$ and $a<b$. If $f:[a,b]\to \R$ is a differentiable function and $|f'(v)|$ is integrable on $[a,b]$, then
\[
   \bigg|\frac{1}{r}\sum_{a<k\leq b} f\Big(\frac{k}{r}\Big)- \int_{a/r}^{b/r}f(v) dv\bigg|\leq \frac{1}{r}\int_{a/r}^{b/r}|f'(v)| dv.
   \]
   \end{lemma}
   
   \begin{proof}
   	The inequality simply follows from the beginning of the proof of Euler-Maclaurin summation formula \cite[p.~5]{GT}.
   	\end{proof}
In what follows, $\bar{Z}$ stands for the vector $(Z_1,\ldots,Z_n)$, where $Z_j$, $j\in\N$, are the Poisson random variables such that $\E Z_j=1/j$.

   \begin{lemma} \label{Prm} For all  $m\geq 0$, we have
\[
\Pr\big(\ell_r(\bar Z)=m\big)=\frac{\re^{-\gamma}}{r}\varrho\Big(\frac{m}{r}\Big)\bigg(1+O\left(\frac{1+{\mathbf{1}\{m>r\}\re^{\xi(m/r)}}}{r}\right)\bigg).
\]
\end{lemma}

\begin{proof} 
If $0\leq m\leq r$, this is an easy exercise.
For $m>r$, apply Theorem~\ref{thm2} and \eqref{P-ell}.
\end{proof}

Here and subsequently, $\xi(v)$, $v\geq 1$, is the function defined in Lemma~\ref{xi}.

\begin{lemma} \label{omegau} The function $\omega(v)-\re^{-\gamma}$, $v\geq1$,
 has  only simple zeros $1\leq \lambda_1<\lambda_2<\cdots$ satisfying
   \[
      \lambda_{k+1}-\lambda_k=1-\frac{1}{\log k}+O\left(\frac{\log\log k}{\log^2 k}\right), \quad k\geq 3.
      \]
   \end{lemma}

\begin{proof}
	 See \cite{AH-JLMS90} or \cite[p.~418]{GT}.
\end{proof}	 
	 
\smallskip

   \begin{lemma} \label{Ru-lema} The functions $R(u)$ and $\vartheta(u)$  introduced in Lemma~\ref{omega,R} are differentiable and
\[
   R'(u)=-R(u)\Big(\xi_0(u)+O\left(\frac{1}{u}\right)\Big),
   \qquad
    \vartheta'(u)=\frac{\pi \xi(u)}{\xi(u)-1}+O\left(\frac{1}{\xi(u)^2}\right)
 \]
 if $u\geq u_0$, where $\xi_0(u)$ is defined in Lemma~\ref{Robert}.
   \end{lemma}

\begin{proof} 
See \cite[Lemma~4]{GT-Crible}.
\end{proof}

   \begin{lemma} \label{Dick} If  $u\geq 1$, then
      \[
\varrho(u-1)=\varrho(u)\re^{\xi(u)}\left(1+O\left(\frac{1}{u}\right)\right).
\]
  Moreover,
 \[
\varrho'(u)=-\xi(u)\varrho(u)\left(1+O\left(\frac{1}{u}\right)\right)
\]
for $u\geq u_0$ and
\begin{equation*}
\varrho(u-v)\ll\varrho(u)\re^{v\xi(u)}
\end{equation*}
uniformly for $0\leq v\leq u$.
\end{lemma}

\begin{proof} 
See \cite[pp.~376 and~377]{GT}.
\end{proof}


\begin{proof}[Proof of Theorem~\ref{EMRP}]
In the case $u\leq u_0$, the argument is much easier; therefore, we concentrate on the proof when  $u> u_0$.  Let us split the sum in expression \eqref{dtv(n,r)andP} of distance $d_{TV}(n,r)$ into four parts:
\begin{align*}
K_1&:=
\sum_{0\leq m<n-r} \Pr\big(\ell_r(\bar Z)=m\big)\bigg|\re^{H_n}\Pr\big(\ell_{rn}(\bar Z)=n-m\big) -1\bigg|,\\
K_2&:=
\sum_{n-r\leq m<n} \Pr\big(\ell_r(\bar Z)=m\big),\\
K_3&:=
\Pr\big(\ell_r(\bar Z)=n\big)\bigg|\re^{H_n}\Pr\big(\ell_{rn}(\bar Z)=0\big) -1\bigg|,\\
K_4&:=
\sum_{m>n} \Pr\big(\ell_r(\bar Z)=m\big).
\end{align*}

We evaluate $K_l$, $1\leq l\leq 4$, in order of increasing difficulty. An estimate of $K_3$ is easy. Indeed,  applying Lemma \ref{Prm}, we have
\begin{align}
K_3&=\re^{-\gamma}\frac{\varrho(u)}{r}\big|\re^{H_r}-1\big|\bigg(1 +O\left(\frac{\re^{\xi(u)}}{r}\right)\bigg)\nonumber\\
&=
\varrho(u)\bigg(1+O\left(\frac{\re^{\xi}}{r}\right)\bigg).
\label{K3}
\end{align}

We infer again from Lemma  \ref{Prm} that
\begin{align*}
K_2&=\sum_{n-r\leq m<n} \Pr\big(\ell_r(\bar Z)=m\big)\\
&=\bigg(\re^{-\gamma}+O\left(\frac{\re^\xi}{r}\right)\bigg)\frac{1}{r}\sum_{n-r\leq m<n}\varrho\Big(\frac{m}{r}\Big).
\end{align*}
Here we have used the fact that $\xi:=\xi(u)$ is monotonically increasing. Summing up the values of monotonically decreasing function $\varrho(u)$, we proceed
\begin{align}
K_2&=
\bigg(\re^{-\gamma}+O\left(\frac{\re^\xi}{r}\right)\bigg)\int_{u-1}^u\varrho(v) dv+O\left(\frac{\varrho(u-1)}{r}\right)\nonumber\\
&=
\int_{u-1}^u\big|\omega(u-v)-\re^{-\gamma}\big| \varrho(v) dv+ O\left(\frac{ u\re^{\xi}\varrho(u)}{r}\right).
\label{K2}
\end{align}
In the last step, we applied the definition of $\omega(v)$, a relation $u\varrho(u)=\int_{u-1}^u \varrho(v) dv$ and Lemma \ref{Dick}.

Similarly, by Lemmas \ref{Prm} and \ref{KH-ineq} (now the use with $b=\infty$ is legitimate), we have
\begin{align}
K_4&= \sum_{m>n} \Pr\big(\ell_r(\bar Z)=m\big)\nonumber\\
&=
\frac{\re^{-\gamma}}{r}\sum_{m>n}\varrho\Big(\frac{m}{r}\Big)+
O\left(\frac{1}{r^2}\sum_{m>n}\varrho\Big(\frac{m}{r}\Big)\re^{\xi(m/r)}\right)\nonumber\\
&=
\re^{-\gamma}\int_{u}^\infty\varrho(v) dv+ O\left(\frac{\varrho(u)}{r}\right)\nonumber\\
&\quad\ +
O\left(\frac{1}{r}\int_{u}^\infty\varrho(v)\re^{\xi(v)} dv\right)+
O\left(\frac{1}{r^2}\int_{u}^\infty\Big|\big(\varrho(v)\re^{\xi(v)}\big)'\Big| dv\right)\nonumber\\
&=
\int_{u}^\infty\big|\omega(u-v)-\re^{-\gamma}\big| \varrho(v) dv+ O\left(\frac{\re^{\xi}\varrho(u)}{r}\right)+ O\left(\frac{1}{r}\int_{u}^\infty\re^{\xi(v)} d(-\varrho(v))\right)\nonumber\\
&=
\int_{u}^\infty\big|\omega(u-v)-\re^{-\gamma}\big| \varrho(v) dv+ O\left(\frac{\re^{\xi}\varrho(u)}{r}\right).
\label{K4}
\end{align}
Estimating the appearing integrals, we have applied a relation $\xi'(v)\sim v^{-1}$ as $v\to\infty$ and Lemma \ref{Prm}.

Integral $K_1$ is crucial. By Lemma \ref{Prm} again, we can start with
\[
K_1=\frac{\re^{-\gamma}}{r}\Big(1+O\left(\frac{\re^{\xi}}{r}\right)\Big)\sum_{0\leq m<n-r}\varrho\Big(\frac{m}{r}\Big)
\Big|\re^{H_n}\Pr\big(\ell_{rn}(\bar Z)=n-m\big)-1\Big|,
\]
where according to Theorem~\ref{theorem3.2} and \eqref{P-ell2},
\begin{align*}
\re^{H_n}\Pr\big(\ell_{rn}(\bar Z)=n-m\big)
&=\re^{\gamma} \omega\Big(\frac{n-m}{r}\Big)\\
&\quad\ +
O\left(\frac{1}{r} \Big(\frac{n-m}{r}\Big)^{3/2}R\Big(\frac{n-m}{r}\Big)\log^2\Big(1+\frac{n-m}{r}\Big)
\right).
\end{align*}
Hence,
\begin{align*}
K_1&=\frac{1}{r}\sum_{0\leq m<n-r}\varrho\Big(\frac{m}{r}\Big)\Big|\omega\Big(\frac{n-m}{r}\Big)-\re^{-\gamma}\Big|\\
&\quad\
+O\left(\frac{u^{3/2} \log^2 (u+1)}{r^2} \sum_{0\leq m<n-r}\varrho\Big(\frac{m}{r}\Big)R\Big(\frac{n-m}{r}\Big)\right)\\
&=:
K_{11}+K_{12}
\end{align*}
by $\re^{\xi}\ll u\log (u+1)$ and Lemma~\ref{omega,R}.

 By virtue of Lemma \ref{omegau}, the difference $\omega(u)-\re^{\gamma}$ changes the sign in the interval $(1, u]$ at most $O\left(u\right)$ times; therefore, Lemma \ref{KH-ineq} can be applied separately for the partial sums in $K_{11}$. Later the estimates can be glued up. Consequently,
\begin{align*}
K_{11}&=
 \int_0^{u-1} \varrho(v)\big|\omega(u-v)-\re^{-\gamma}\big| dv \\
 &\quad\ +
 O\left(\frac{1}{r}\int_0^{u-1} \Big| d\big(\varrho(v)\big|\omega(u-v)-\re^{-\gamma}\big|\big)\Big|\right)\\
 &=
 \int_0^{u-1} \varrho(v)\big|\omega(u-v)-\re^{-\gamma}\big| dv \\
 &\quad\
+ O\left(\frac{1}{r}\int_0^{u-1} |\varrho'(v)|\big|\omega(u-v)-\re^{-\gamma}\big| dv\right)\\
 &\quad\
+ O\left(\frac{1}{r}\int_1^{u} \varrho(u-v)\Big|d\big|\omega(v)-\re^{-\gamma}\big|\Big|\right).
\end{align*}
As it has been noticed in \cite{GT-Crible} on p. 26, the definitions of $\omega(v)$ and $R(v)$ combined with Lemma \ref{Ru-lema} imply
\[
\Big|d\big|\omega(v)-\re^{-\gamma}\big|\Big|\ll R(v)\log (v+1) dv, \quad v\geq1.
\]
Using also Lemma \ref{Dick}, we obtain
\begin{align*}
K_{11}&=
\int_0^{u-1} \varrho(v)\big|\omega(u-v)-\re^{-\gamma}\big| dv\\
&\quad\ +
 O\left(\frac{\log (u+1)}{r}\int_1^{u} \varrho(u-v) R(v) dv\right).
\end{align*}
The same argument gives
\[
K_{12}\ll \frac{u^{3/2} \log^2 (u+1)}{r}\int_1^{u} \varrho(u-v) R(v) dv
\]
and, further,
\begin{align}
K_1&=
\int_0^{u-1} \varrho(v)\big|\omega(u-v)-\re^{-\gamma}\big| dv\nonumber\\
&\quad\ +
O\left(\frac{u^{3/2} \log^2 (u+1)}{r}\int_1^{u} \varrho(u-v) R(v) dv\right)\nonumber\\
&=
\int_0^{u-1} \varrho(v)\big|\omega(u-v)-\re^{-\gamma}\big| dv +
O\left(\frac{u^{3/2} \log^2 (u+1)}{r} H(u)\right)
\label{K1}
\end{align}
as it has been observed on p. 25 of \cite{GT-Crible} without giving any hint. The last estimate can be substantiated as follows.

   The function
   \[
     g(v):=\big|\omega(v)-\re^{-\gamma}\big| R(v)^{-1}
   \]
   is uniformly continuous and, by Lemma \ref{omegau},  has only $O\left(u\right)$ simple separated  zeros in $[1, u]$. Thus, given $0<\varepsilon<1$, the Lebesgue measure
      \[
         \operatorname{meas} A_\varepsilon:=\operatorname{meas}\{v:\; 1\leq v\leq u,\,  g(v)\leq \varepsilon\}\ll u\delta(\varepsilon)
         \]
      with a $\delta(\varepsilon)\to0$ as $\varepsilon\to0$. Now we can fix $\varepsilon$, so that
      \[
          L(u):=\int_1^{u} \varrho(u-v) R(v) dv\geq 2\int_{A_\varepsilon} \varrho(u-v) R(v) dv=:2L_\varepsilon(u).
          \]
On the other hand,
          \[
          L(u)\leq   L_\varepsilon(u)+\varepsilon^{-1} \int_1^u \varrho(u-v)\big|\omega(v)-\re^{-\gamma}\big| dv.
          \]
       Hence, we obtain an inequality
    \[
          L(u)\leq(1/2)L(u)+ \varepsilon^{-1}H(u)
          \]
implying (\ref{K1}).

        By the asymptotic formula of $H(u)$ given in (\ref{Hu-asymp}), the error terms in estimates (\ref{K3}), (\ref{K2}) and (\ref{K4}) are satisfactory; therefore, inserting them together with (\ref{K1}) into \eqref{dtv(n,r)andP} we complete the proof of Theorem~\ref{EMRP}.
        
\end{proof}        

\chapter{Concluding remarks}\label{Chapter6}	

\begin{itemize}
\item The classic saddle-point method is effective in obtaining asymptotic formulas for the density of permutations
with a given cycle structure pattern in a symmetric group but not in the whole range of variables.

\item The use of the Laplace transforms of the functions involved in the main asymptotic terms supplements the classic method.

\item The technique developed in number theory for the analysis of prime factorization of a natural number can be used to obtain similar results for the cyclic structure
      of a permutation $\sigma\in\mathrm{S}_n$.
      
\item The results for the third problem can be refined by improving the estimates obtained for the densities $\nu(n,r)$ and $\kappa(n,r)$.      
       	
\item The applied techniques are effective and can be used to solve similar problems on other decomposable combinatorial structures. In particular, the methods should be efficient when employed to the logarithmic combinatorial structures described in~\cite{ABT}.	
\end{itemize}

\newpage
\null\thispagestyle{empty}
\newpage

\begin{center}
	\textbf{\Large Short CV}
\end{center}

\begin{center}
	\textbf{\large Robertas Petuchovas}
\end{center}

\bigskip

\begin{description}
		\item[Birth Date and Place] \hfill\\
		\vspace{-30pt}
		\begin{description}
			\item[27/01/1988] Klaipėda city, Lithuania.
		\end{description}
		
	\item[Education] \hfill\\
	\vspace{-30pt}
	\begin{description}
		\item[1998] Graduate from Klaipėda primary school "Aušrinė".
		\item[2006] Graduate from Klaipėda high school "Ąžuolyno".
		\item[2010] Bachelor's degree in Mathematics and Applications of Mathematics 
		
		 \hspace{+1mm} from Vilnius University.
		\item[2012] Master's degree in Mathematics from Vilnius University.
		\item[2016] PhD's degree in Mathematics from Vilnius University.
	\end{description}
	
	\item[Teaching Experience] \hfill\\
	\vspace{-30pt}
	\begin{description}
		\item[2012-2016] Teaching fellow at Vilnius University.
	\end{description}
	\item[Other Work Experience] \hfill\\
	\vspace{-30pt}
	\begin{description}
		\item[2010] \text{(three months)} \mbox{Assistant of Economic Department at "Šilutės Durpės" Ltd.} 
		\item[2011] \text{(six months)} Vilnius Regional manager at "Dėvėdra" Ltd.
	\end{description}
	
		\item[Mother Tongues] \hfill\\
		\vspace{-30pt}
		\begin{description}
			\item Lithuanian, Russian.
		\end{description}

			\item[Hobbies] \hfill\\
			\vspace{-30pt}
			\begin{description}
				\item[2012] Lithuanian amateur chess championship, third place winner.
				\item[2015-2016]\mbox{Long-distance running competition participant.}
			\end{description}
\end{description}

\end{document}